\documentclass[11pt,a4paper,reqno]{amsart}
\usepackage[hidelinks]{hyperref}
\pdfoutput=1
\usepackage{amssymb,interval,eulervm,palatino,fullpage,enumitem,amsaddr}
\usepackage[USenglish]{babel}
\usepackage[justification=centering,skip=5pt]{caption} 
\usepackage{tikz,tikz-cd}
\usetikzlibrary{arrows,arrows.meta}
\setenumerate{label=(\roman*),itemsep=2pt,topsep=2pt}

\allowdisplaybreaks[4]
\newcommand{\doi}[1]{\url{https://doi.org/#1}}
\newcommand{\isbn}[1]{\url{https://isbnsearch.org/isbn/#1}}
\newcommand{\arxiv}[1]{\href{https://arxiv.org/abs/#1}{preprint arXiv:#1}}
\newcommand{\web}[1]{\url{#1}}

\renewcommand{\emptyset}{\varnothing}
\numberwithin{equation}{section}

\newtheorem{thm}{Theorem}[section]
\newtheorem{lemma}[thm]{Lemma}
\newtheorem{prop}[thm]{Proposition}
\newtheorem{rem}[thm]{Remark}

\newtheorem{df}[thm]{Definition}
\newtheorem{ex}[thm]{Example}
\newtheorem{nex}[thm]{Non-Example}

\newcommand{\C}{\mathbb{C}}
\newcommand{\Z}{\mathbb{Z}}
\newcommand{\N}{\mathbb{N}}
\newcommand{\R}{\mathbb{R}}
\newcommand{\inner}[1]{\left<#1\right>}
\newcommand{\ket}[1]{\left|#1\right>}
\newcommand{\mv}[1]{\smash[b]{\underline{#1}}}
\newcommand{\id}{\mathrm{id}}

\newcommand{\verteq}{\rotatebox{90}{$\,=$}}
\newcommand{\equalto}[2]{\underset{\displaystyle\overset{\mkern3mu\verteq}{#2}}{#1}}

\makeatletter
\def\@tocline#1#2#3#4#5#6#7{\relax
  \ifnum #1>\c@tocdepth 
  \else
    \par \addpenalty\@secpenalty\addvspace{#2}%
    \begingroup \hyphenpenalty\@M
    \@ifempty{#4}{%
      \@tempdima\csname r@tocindent\number#1\endcsname\relax
    }{%
      \@tempdima#4\relax
    }%
    \parindent\z@ \leftskip#3\relax \advance\leftskip\@tempdima\relax
    \rightskip\@pnumwidth plus4em \parfillskip-\@pnumwidth
    #5\leavevmode\hskip-\@tempdima
      \ifcase #1
       \or\or \hskip 1em \or \hskip 2em \else \hskip 3em \fi%
      #6 \hskip 0.5em \nobreak\relax
    \dotfill\hbox to\@pnumwidth{\@tocpagenum{#7}}\par
    \nobreak
    \endgroup
  \fi}
\makeatother

\begin{document}
\leftmargini=2em

\title{\vspace*{-1cm}Quantum spheres as graph C*-algebras: a review}

\author[F.~D'Andrea]{\vspace*{-5mm}Francesco D'Andrea}

\address{Dipartimento di Matematica e Applicazioni ``R.~Caccioppoli'' \\ Universit\`a di Napoli Federico II \\
Complesso MSA, Via Cintia, 80126 Napoli, Italy}



\begin{abstract}
In this survey, we discuss the description of Vaksman-Soibelman quantum spheres
using graph C*-algebras, following the seminal work of Hong and Szyma{\'n}ski.
We give a slightly different proof of the isomorphism with a graph C*-algebra, borrowing the idea of Mikkelsen and Kaad of using conditional expectations to prove the desired result.
\end{abstract}

\maketitle

\begin{center}
\begin{minipage}{0.8\textwidth}
\parskip=0pt\small\tableofcontents
\end{minipage}
\end{center}

\bigskip

\section{Introduction}
Quantum spheres are among the most studied examples of quantum homogeneous spaces.
The first examples are due to Podle\'s \cite{Pod87}, who introduced his well-known 2-dimensional examples as homogeneous spaces for Woronowicz's quantum $SU(2)$ group \cite{Wor80}. A few years later, Vaksman and Soibelman introduced quantum spheres of arbitrary odd dimension \cite{VS91} as homogeneous spaces of higher dimensional quantum unitary groups \cite{Wor87}.
They are described by a $\Z$-graded algebra whose degree $0$ part we interpret as algebra of functions on a quantum complex projective space (for a review of the geometry of quantum projective spaces, one can see e.g.~\cite{DL12}).
A main breakthrough in their study was the realization that the C*-enveloping algebras are graph C*-algebras \cite{HS02}.

The aim of this survey is to give a detailed proof of the isomorphism between the universal C*-algebra of a Vaksman-Soibelman quantum sphere and a graph C*-algebra. 
Rather than the original proof in \cite{HS02}, here we follow the idea of Mikkelsen and Kaad \cite{MK22} of using conditional expectations to prove the desired isomorphism.

Historically, the $2n+1$ dimensional quantum sphere $S^{2n+1}_q$ was introduced as a quantum homogeneous spaces of $SU_q(n+1)$. In this review, however, we avoid on purpose talking about $SU_q(n+1)$ because we want to make the paper self-contained, and we prefer not to rely on results about $SU_q(n+1)$ to prove the desired results about quantum spheres.
The paper is intended to be a pedagogical review, and it is written in an elementary style in order to make it more accessible to a wide audience.

The plan of the paper is as follows.
In Sect.~\ref{sec:2}, we recall some basic facts about group actions on C*-algebras and conditional expectations (Sect.~\ref{sec:21}), and universal C*-algebras (Sect.~\ref{sec:22}). Those who are familiar with \mbox{C*-algebras} can readily skip this part. We also briefly recall the notions of Leavitt path algebra and of graph C*-algebra (Sect.~\ref{sec:4}).
In Sect.~\ref{sec:3}, we introduce the coordinate algebra $\mathcal{A}(S^{2n+1}_q)$ of the $(2n+1)$-dimensional quantum sphere. We start in Sect.~\ref{sec:alg} by discussing some of its basic algebraic properties. Next, in Sect.~\ref{sec:5}, we discuss two realizations of $\mathcal{A}(S^{2n+1}_q)$, for $q=0$, as the Leavitt path algebra of the two graphs $\Sigma_n$ and $\widetilde{\Sigma}_n$ shown in Figures \ref{fig:sphere} and \ref{fig:sphereB}. We also show that, for all values of $q$, the canonical *-homomorphism from $\mathcal{A}(S^{2n+1}_q)$ to its C*-enveloping algebra $C(S^{2n+1}_q)$ is injective. In Sect.~\ref{sec:6}, we pass to the case of positive parameter and prove the isomorphism $C(S^{2n+1}_q)\cong C(S^{2n+1}_0)$ for all \mbox{$0<q<1$}.

We should mention that the first proof that the isomorphism class of $C(S^{2n+1}_q)$ is independent of $q$ is in \cite{She97bis,She97}. The one proposed here is an independent proof inspired by the ideas of \cite{MK22}. The interplay between the graph and groupoid picture in \cite{She97bis,She97} is discussed in \cite{Dan24a}. Contrary to what happens with C*-enveloping algebras, the polynomial algebra $\mathcal{A}(S^{2n+1}_q)$ does depend on the value of the deformation parameter, see \cite{Dan24b}.

\medskip

\noindent
{\bf Notations and conventions.} In this paper $q$ is a real number and we assume that
\[
0\leq q<1 .
\]
When we write $q\neq 0$, we always mean $0<q<1$.

We denote by $\N$ the set of natural numbers (including $0$) and by $\overline{\N}=\N\cup\{+\infty\}$ its one-point compactification. Note that, for every fixed $q\neq 0$, the map $k\mapsto q^k$ is a closed embedding of $\overline{\N}$ into $\R$, if we adopt the convention $q^{+\infty}=0$. We denote by either $\mathbb{T}$ or $S^1$ the set of unitary complex numbers, and by $\mathbb{T}^n$ the $n$-torus.

Every algebra in this paper is associative and over the complex numbers. In the case of unital algebras, homomorphisms and representations are assumed to be unital. By a Hilbert space we always mean a complex one.

We adopt the convention that an empty sum is $0$ and an empty product is $1$.
By a group action on a (*-)algebra, we will always mean one by (*-)automorphisms.
If $A$ is an algebra and $\alpha:G\to\mathrm{Aut}(A)$ an action of a group $G$, we  denote by $A^G$ the subalgebra of $G$-invariant elements of $A$.
Note that if $A$ is a C*-algebra and $\alpha:G\to\mathrm{Aut}(A)$ is \emph{any} group action (not necessarily continuous), since *-homomorphism between C*-algebras are continuous, $A^G$ is automatically a \mbox{C*-subalgebra} of $A$.
Indeed, if $(a_n)$ is a sequence in $A^G$ and $a_n\to a\in A$, for all $g\in G$ one has
\[
\alpha_g(a)=\alpha_g(\lim_{n\to\infty}a_n)=\lim_{n\to\infty}\alpha_g(a_n)=\lim_{n\to\infty}a_n=a ,
\]
which means that $a\in A^G$.

If $H$ is a Hilbert space, we denote by $\mathcal{B}(H)$ the set of bounded operators and by $\mathcal{U}(H)$ the set of unitary operators on $H$.

\section{Mathematical preliminaries}\label{sec:2}

\subsection{Group actions and conditional expectations}\label{sec:21}
Recall that an action $\alpha$ of a topological group $G$ on a topological space $X$ is called \emph{continuous} if the map $G\times X\to X$, $(g,x)\mapsto\alpha_g(x)$, is continuous. It is called \emph{strongly continuous} if, for all $x\in X$, the map $G\to X$, $g\mapsto\alpha_g(x)$, is continuous. If $X$ is a C*-algebra (resp.~a Hilbert space), by a strongly continuous action (resp.~representation) we will always mean with respect to the norm topology on $X$.

Clearly continuity implies strong continuity, but the converse is in general not true (despite the name, strong continuity is a weaker notion than continuity). It is true in special cases, like the one considered in the next lemma.

\begin{lemma}\label{lemma:cGa}
Let $X$ be a normed vector space and $\alpha$ an isometric action of a topological group $G$ on $X$.
Then $\alpha$ is continuous if and only if it is strongly continuous.
\end{lemma}

\begin{proof}
We need to prove the implication $\Leftarrow$. By hypothesis, $\|\alpha_g(x)\|=\|x\|$ for all $(g,x)\in G\times X$.
Let $(g_\lambda,x_\lambda)_{\lambda\in\Lambda}$ be a convergent net in $G\times X$, with limit $(g,x)$. Then
\begin{align*}
\|\alpha_{g_\lambda}(x_\lambda)-\alpha_g(x)\| &=
\|\alpha_{g_\lambda}(x_\lambda-x)+\alpha_{g_\lambda}(x)-\alpha_g(x)\| \\
&\leq
\|\alpha_{g_\lambda}(x_\lambda-x)\|+\|\alpha_{g_\lambda}(x)-\alpha_g(x)\|
=\|x_\lambda-x\|+\|\alpha_{g_\lambda}(x)-\alpha_g(x)\| .
\end{align*}
The first term goes to $0$ because $x_\lambda\to x$, the second because of strong continuity. Thus, one has $\alpha_{g_\lambda}(x_\lambda)\to\alpha_g(x)$, i.e.~$\alpha:G\times X\to X$ is continuous.
\end{proof}

Since every group action on a C*-algebra is isometric (injective *-homomorphisms between C*-algebras are always isometric), as a corollary we have an equivalence between continuity and strong continuity for actions of topological groups on C*-algebras.
Similarly, a unitary representation of a topological group is continuous if and only if it is strongly continuous.

Let us stress that continuity of an action $\alpha:G\times A\to A$ on a C*-algebra doesn't mean that the map $A\to\mathrm{Aut}(A)$, $g\mapsto\alpha_g$, is continuous w.r.t.~the uniform topology on $\mathrm{Aut}(A)$. Similarly, continuity of a unitary representation $U:G\times H\to H$ doesn't mean that the map $G\to\mathcal{U}(H)$, $g\mapsto U_g$, is continuous w.r.t.~the norm-topology on the codomain (the latter is a stronger condition), nor does it mean (a priori) that the adjoint action $\alpha_g(b):=U_gbU_g^*$ on $\mathcal{B}(H)$ is strongly continuous.

\medskip

A unitary representation $U$ of a topological group $G$ on a Hilbert space $H$ is called \emph{weakly continuous} if, for all $v,w\in H$, the map $G\to\C$ given by $g\mapsto\inner{v,U_gw}$ is continuous. Thus, $U$ is weakly (resp.~strongly) continuous if as a map $G\to\mathcal{U}(H)$ is continuous w.r.t.~the weak (resp.~strong) operator topology on $\mathcal{U}(H)$.

\begin{lemma}\label{lemma:sc}
For a unitary representation $U$ of a topological group $G$ on a Hilbert space $H$, the following conditions are equivalent
\begin{enumerate}[label=(\roman*)]
\item\label{en:lemA} $U$ is strongly continuous,
\item\label{en:lemB} $U$ is weakly continuous.
\end{enumerate}
\end{lemma}

\begin{proof}
For $v,w\in H$ and $b\in\mathcal{B}(H)$ we have two maps
\begin{align*}
f_w  & : G\longrightarrow H, && g\longmapsto U_gw , \\
f_{v,w} & : G\longrightarrow\C, && g\longmapsto\inner{v,U_gw} .
\end{align*}
\ref{en:lemA} means that all maps $f_w$ are continuous,
\ref{en:lemB} that all maps $f_{v,w}$ are continuous.
The implication \ref{en:lemA}$\,\Rightarrow\,$\ref{en:lemB}
is obvious, since $f_{v,w}$ is the composition of $f_w$ with the (continuous) map $H\to\C$ given by the scalar product with $v$.
We must prove \ref{en:lemB}$\,\Rightarrow\,$\ref{en:lemA}.

Let $(g_\lambda)_{\lambda\in\Lambda}$ be a net in $G$ convergent to $g\in G$.
If $U$ is weakly continuous, for every $w\in H$ and writing $v=U_gw$, we have
\[
\inner{U_gw,U_{g_\lambda}w}=f_{v,w}(g_\lambda) \longrightarrow f_{v,w}(g)=\|U_gw\|^2=\|w\|^2 ,
\]
and then
\[
\|f_w(g_\lambda)-f_w(g)\|^2=\|U_{g_\lambda}w-U_gw\|^2=2\|w\|^2-2\,\Re\inner{U_gw,U_{g_\lambda}w}\longrightarrow 0. \qedhere
\]
\end{proof}

Recall that a linear map $f:A\to B$ between C*-algebras is called \emph{faithful} if, for all positive elements $a\in A$, $f(a)=0$ implies $a=0$. A *-homomorphism is faithful if and only if it is injective, for if $f$ is faithful, $a\in A$ and $f(a)=0$,
then $f(a^*a)=f(a)^*f(a)=0$, which implies $a^*a=0$ by faithfulness, and then $a=0$ by the C*-identity.
In particular,

\begin{rem}
A bounded *-representation of a C*-algebra is faithful if and only if it is injective.
\end{rem}

A \emph{C*-dynamical system} is a triple $(A,G,\alpha)$, where $A$ is a C*-algebra, $G$ a locally compact Hausdorff group, and $\alpha$ a strongly continuous action of $G$ on $A$.
A \emph{covariant representation} of a C*-dynamical system $(A,G,\alpha)$ is a triple $(H,\pi,U)$ consisting of a Hilbert space $H$, a bounded *-representation $\pi$ of $A$ on $H$, and a strongly continuous unitary representation $U$ of $G$ on $H$, such that
$U_g\pi(a)U_g^*=\pi(\alpha_ga)$ for all $a\in A$ and $g\in G$.

If $A$ is a C*-algebra and $B$ a C*-subalgebra, a linear map $\vartheta:A\to B$ is called a \emph{conditional expectation} if it is $B$-bilinear, idempotent ($\vartheta(b)=b\;\forall\;b\in B$), completely positive and a contraction. The important property for us will be positivity.
By Tomiyama's theorem (see e.g.~\cite[Thm.~1.5.10]{BO08}), every contractive idempotent map $A\to B$ (with $B\subseteq A$ a C*-subalgebra) is a conditional expectation. Notice that, since $\vartheta$ is idempotent, it is also surjective.

\begin{lemma}\label{lemma:faithcond}
Let $(A,G,\alpha)$ be a C*-dynamical system, with $G$ compact, and let $\mu$ be the normalized Haar measure on $G$. Then,
\begin{enumerate}
\item\label{en:faithcondA}
The map $\vartheta:A\to A^G$ given by
\begin{equation}\label{eq:conditional}
\vartheta(a):=\int_G\alpha_g(a)d\mu_g ,\qquad a\in A,
\end{equation}
is a faithful conditional expectation.

\item\label{en:faithcondB}
Let $(H,\pi,U)$ be a covariant representation of $(A,G,\alpha)$. If the restriction $\pi|_{A^G}:A^G\to\mathcal{B}(H)$ is faithful, then $\pi$ is faithful.
\end{enumerate}
\end{lemma}

\begin{proof}
\ref{en:faithcondA}
Let $a\in A$.
By invariance of the Haar measure, for all $g'\in G$,
\[
\alpha_{g'}\vartheta(a)=\int_G\alpha_{g'g}(a)d\mu_g=\int_G\alpha_{g''}(a)d\mu_{g''}=\vartheta(a) ,
\]
where $g''=g'g$. Thus, $\vartheta(a)$ is $G$-invariant, i.e.~it belongs to $A^G$. By the choice of normalization, $\vartheta$ is also idempotent: for every $a\in A^G$, $\vartheta(a)=\int_G\alpha_g(a)d\mu_g=\int_Gad\mu_g=a\int_Gd\mu_g=a$.
Finally, since
\[
\|\vartheta(a)\|\leq \int_G \|\alpha_g(a)\|d\mu_g=\|a\| ,
\]
the map $\vartheta$ is norm decreasing. By Tomiyama's theorem, this implies that $\vartheta$ is a conditional expectation. To prove faithfulness, let now $a\in A$ be a non-zero positive element and let $\omega$ be a state on $A$
such that $\omega(a)=\|a\|$. Since $g\mapsto\omega(\alpha_g(a))$ is continuous, there exists a non-empty open set $V\subseteq G$ such that $\omega(\alpha_g(a))\geq\|a\|/2$ for all $g\in V$. But
\[
\omega(\vartheta(a))=\int_G\omega(\alpha_g(a))d\mu_g\geq
\int_V\omega(\alpha_g(a))d\mu_g\geq \mu(V)\cdot\|a\|/2 .
\]
Since the Haar measure $\mu(V)$ of a non-empty open set is non-zero, one has $\omega(\vartheta(a))\neq 0$, and then $\vartheta(a)\neq 0$.

\medskip

\noindent\ref{en:faithcondB}
By point \ref{en:faithcondA} of this lemma, the map $\vartheta$ in \eqref{eq:conditional} is a faithful and positive.
Let $\widetilde{\vartheta}:\mathcal{B}(H)\to \mathcal{B}(H)$ be the linear map defined by
\[
\widetilde{\vartheta}(b):=\int_{G}U_gbU_g^*d\mu_g
\]
for all $b\in\mathcal{B}(H)$. By covariance of the representation,
\begin{equation}\label{eq:faithpt}
\pi\circ\vartheta=\widetilde\vartheta\circ\pi .
\end{equation}
Let $a\in A$ be a positive element in the kernel of $\pi$. From \eqref{eq:faithpt} we get
\[
\pi\big(\vartheta(a)\big)=\widetilde\vartheta\big(\pi(a)\big)=\widetilde\vartheta(0)=0 .
\]
Since $\vartheta(a)$ is a positive element in $A^G$ and $\pi|_{A^G}$ is faithful, we deduce that $\vartheta(a)=0$. Since $\vartheta$ is faithful, it follows that $a=0$. Hence $\pi$ is faithful.
\end{proof}

\subsection{Universal C*-algebras}\label{sec:22}

In this section we restrict our attention to unital algebras and unital homomorphisms (it is important to specify this, since universal constructions depend on the ambient category). Thus, every algebra in this section is assumed to be (complex, associative and) \emph{unital}, even if not stated explicitly.

\begin{df}\label{df:universal}
Let $A$ be a *-algebra, $B$ a C*-algebra and $\phi:A\to B$ a *-homomorphism. The pair $(B,\phi)$ is called a \emph{universal C*-algebra} of $A$ if for every C*-algebra $C$ and every *-homomorphism $f:A\to C$ there exists a unique *-homomorphism $\widetilde{f}:B\to C$ making the following diagram commute:
\[
\begin{tikzpicture}[>=To,xscale=2.5,yscale=2]

\node (a) at (0,1) {$A$};
\node (b) at (1,1) {$B$};
\node (c) at (1,0) {$C$};

\draw[font=\footnotesize,->]
	(a) edge node[below left,pos=0.45] {$f$} (c)
	(a) edge node[above] {$\phi$} (b)
	(b) edge[dashed] node[right,pos=0.45] {$\widetilde{f}$} (c);

\end{tikzpicture}
\]
\end{df}

In categorical language, denoting by $F$ the forgetful functor from C*-algebras to *-algebras,
the pair $(B,\phi)$ in the above definition is an initial object in the comma category $(A\downarrow F)$, and is
what we call a \emph{universal morphism} from $A$ to $F$, cf.~\cite[Cap.~3]{Mac10}.

One can see \cite{Bla85} for a study of universal C*-algebras described by generators and (admissible) relations.
As usual with universal morphisms, a universal C*-algebra of $A$ may not exist (see the next non-example), but if it exists it is essentially unique (see the next Prop.~\ref{prop:unique}).

\begin{nex}
The algebra $A:=\C[x]$, with *-structure $x^*=x$, has no universal C*-algebra. By contradiction, suppose that $(B,\phi)$ satisfies the universal property in Def.~\ref{df:universal}.
For every $\lambda\in\R$ one has a *-homomorphism $\C[x]\to\C$ that maps $x$ to $\lambda$. By the universal property, there must exist a *-homomorphism $B\to\C$ that maps $\phi(x)$ to $\lambda$. Since *-homomorphisms between C*-algebras are norm decreasing, we deduce that $\|\phi(x)\|\geq\lambda$ for all $\lambda\in\R$, which is clearly impossible.
\end{nex}

\begin{prop}\label{prop:unique}
Let $(B,\phi)$ and $(B',\phi')$ be universal C*-algebras of the same *-algebra $A$. Then, there exists a unique isomorphism $f:B\to B'$ such that $\phi=\phi'\circ f$.
\end{prop}

\begin{proof}
This is a standard proof which is valid for arbitrary universal morphisms. We repeat it here for the reader's ease.
By the universal property, if an isomorphism $f$ as above exists, it is unique. We must show its existence.

By the universal property of both $(B,\phi)$ and $(B',\phi')$, there exists a unique map $\widetilde{\phi}$ and a unique map $\widetilde{\phi}'$ making the following diagram commute
\[
\begin{tikzpicture}[>=To]

\node (a) at (1.5,1.8) {$A$};
\node (b) at (0,0) {$B$};
\node (c) at (3,0) {$B'$};

\draw[font=\footnotesize,->]
	(a) edge node[above right] {$\phi'$} (c)
	(a) edge node[above left] {$\phi$} (b)
	(b) edge[transform canvas={yshift=2pt}] node[above] {$\;\widetilde{\phi}'$} (c)
	(c) edge[transform canvas={yshift=-2pt}] node[below] {$\;\widetilde{\phi}\phantom{'}$} (b);

\end{tikzpicture}
\]
Commutativity of the diagram tells us that $\widetilde{\phi}\circ (\widetilde{\phi}'\circ\phi)=\widetilde{\phi}\circ\phi'=\phi$, thus we have a commutative diagram
\[
\begin{tikzpicture}[>=To]

\node (a) at (1.5,1.8) {$A$};
\node (b) at (0,0) {$B$};
\node (c) at (3,0) {$B$};

\draw[font=\footnotesize,->]
	(a) edge node[above right] {$\phi$} (c)
	(a) edge node[above left] {$\phi$} (b)
	(b) edge[transform canvas={yshift=2pt}] node[above] {$\widetilde{\phi}\circ\widetilde{\phi}'$} (c)
	(b) edge[transform canvas={yshift=-2pt}] node[below] {$\id_B$} (c);

\end{tikzpicture}
\]
From the universal property we deduce that $\widetilde{\phi}\circ\widetilde{\phi}'=\id_B$. Similarly one shows that $\widetilde{\phi}'\circ\widetilde{\phi}=\id_{B'}$, so that $f=\widetilde{\phi}'$ is the desired isomorphism.
\end{proof}

Due to its essential uniqueness, from now on we will talk about \emph{the} universal C*-algebra of $A$ (if any) and denote it by $(A_u,\phi_u)$.

Observe that every *-algebra $A$ has (at least) one bounded *-representation given by the zero representation.
The next definition \eqref{eq:Csemi} then makes sense.

\begin{prop}
Let $A$ be a *-algebra and, for $a\in A$, call
\begin{equation}\label{eq:Csemi}
\|a\|_u:=\sup\big\{\|\rho(a)\|:\rho\text{ is a bounded *-representation of }A\big\}.
\end{equation}
The universal C*-algebra of $A$ exists if and only if $\|a\|_u$ is finite for all $a\in A$.
\end{prop}

\begin{proof}
($\Rightarrow$) Let $(A_u,\phi_u)$ be the universal C*-algebra of $A$ (by hypothesis, it exists).
By the universal property, for every bounded *-representation $\rho:A\to\mathcal{B}(H)$ we have a commutative diagram
\[
\begin{tikzpicture}[>=To,xscale=2.5,yscale=2]

\node (a) at (0,1) {$A$};
\node (b) at (1,1) {$A_u$};
\node (c) at (1,0) {$\mathcal{B}(H)$};

\draw[font=\footnotesize,->]
	(a) edge node[below left,pos=0.45] {$\rho$} (c)
	(a) edge node[above,pos=0.55] {$\phi_u$} (b)
	(b) edge[dashed] node[right,pos=0.45] {$\pi$} (c);

\end{tikzpicture}
\]
where $\pi$ is a bounded *-representation of $A_u$. Since *-homomorphisms between C*-algebras are contractive, for all $a\in A$ one has
\[
\|\rho(a)\|=\|\pi(\phi_u(a))\|\leq \|\phi_u(a)\| ,
\]
where the first two norms are the operator norms in $\mathcal{B}(H)$ and the third one is the C*-norm of $A_u$.
The latter is independent of the representation, hence $\|a\|_u\leq\|\phi_u(a)\|<\infty\;\forall\;a\in A$.

\medskip

\noindent
($\Leftarrow$)
The proof is constructive. If $\|a\|_u$ is finite for all $a\in A$, then the map $a\mapsto\|a\|_u$ is a C*-seminorm, called the \emph{universal C*-seminorm} of $A$.
Putting $J:=\big\{a\in A:\|a\|_u=0\big\}$ the kernel of the seminorm (a closed two-sided *-ideal in $A$), the C*-completion $A_u$ of $A/J$ in the norm induced by $\|\,\cdot\,\|_u$ is well-defined and equipped with a canonical *-homomorphism $\phi_u:A\to A_u$.
This pair $(A_u,\phi_u)$ is called the \emph{C*-enveloping algebra} of $A$.

We now show that it satisfies the universal property. Let $f:A\to C$ be a *-homomorphism to a C*-algebra. Think of $C$ as a concrete C*-algebra of bounded operators on a Hilbert space (by Gelfand theorem) and of $f$ as a bounded *-representation. By definition of the ideal $J$ above, its elements are in the kernel of any bounded *-representation, thus $J\subseteq\ker(f)$. It follows that the map $f:A\to C$ factors through the canonical projection $\pi:A\to A/J$ and we have a commutative diagram
\[
\begin{tikzpicture}[>=To]

\node (a) at (0,0) {$A$};
\node (aj) at (2,0) {$A/J$};
\node (au) at (4,0) {$A_u$};
\node (c) at (2,-2) {$C$};

\draw[font=\footnotesize]
	(a) edge[->>] node[above] {$\pi$} (aj)
	(a) edge[->] node[below left,pos=0.45] {$f$} (c)
	(aj) edge[right hook->] (au)
	(aj) edge[->] node[right,pos=0.45] {$f_0$} (c)
	(a) edge[bend left,->] node[above] {$\phi_u$} (au)
	(au) edge[dashed,->] node[below right,pos=0.45] {$\widetilde{f}$} (c);

\end{tikzpicture}
\]
The map $f_0$ is uniquely determined by $f$, since $\pi$ is surjective.
By definition of the universal C*-seminorm one has $\|f_0(\pi(a))\|=\|f(a)\|\leq\|a\|_u$, thus $f_0$ is Lipschitz continuous, and hence uniformly continuous. A uniformly continuous function $D\to Y$ from a dense subset $D$ of a complete metric space $X$ to a metric space $Y$ admits a unique continuous extension $X\to Y$. There exists then a continuous map $\widetilde{f}:A_u\to C$
extending $f_0$. Since $f_0$ is a *-homomorphism and $A/J$ is dense in $A_u$, $\widetilde{f}$ is a *-homomorphism.
Any other *-homomorphism $A_u\to C$ making the above diagram commute must coincide with $\widetilde{f}$, since it is equal to $\widetilde{f}$ on a dense subset of the domain.
\end{proof}

Notice that the map $\phi_u$ is not always injective.
If it is injective, we will say that $A$ \emph{embeds} into its universal C*-algebra, think of $A$ as a dense *-subalgebra of $A_u$, and omit the map $\phi_u$.

\begin{prop}
Assume that the universal C*-algebra of $A$ exists. Then, $\phi_u$ is injective if and only if
$A$ has an injective bounded *-representation.
\end{prop}

\begin{proof}
($\Rightarrow$) Every C*-algebra has a faithful representation (Gelfand-Naimark theorem). Let $\pi$ be a faithful representation of $A_u$. Then $\pi\circ\phi_u$ is an injective bounded *-representation of $A$.

\medskip

\noindent
($\Leftarrow$)
If $A$ has an injective bounded *-representation, then the universal C*-seminorm is a norm, $J=0$ and the map from $A\cong A/J$ to its C*-enveloping algebra is injective.
\end{proof}

\begin{ex}
Let $n\geq 2$ and $A:=\C[x]/(x^n)$ be the truncated polynomial algebra in one real indeterminate, $x=x^*$. Let $f:A\to C$ be a *-homomorphism to another C*-algebra. Then $f(x)^n=f(x^n)=0$ in $C$, which by the C*-identity implies $f(x)=0$. Thus, $f$ factors through the morphism $A\to \C$ defined by $x\mapsto 0$, proving that the universal C*-algebra of $A$ is $A_u=\C$. The canonical *-homomorphism $\phi_u:A\to\C$, $x\mapsto 0$, is obviously not injective.
\end{ex}

\begin{lemma}\label{lemma:autA}
Let $A$ be a *-algebra and $(A_u,\phi_u)$ its universal C*-algebra (assume that it exists). For every $\alpha\in\mathrm{Aut}(A)$ there exists a unique $\widetilde\alpha\in\mathrm{Aut}(A_u)$ such that the following diagram commutes
\[
\begin{tikzpicture}[>=To]

\node (a) at (0,1.8) {$A$};
\node (d) at (2.5,1.8) {$A$};
\node (b) at (0,0) {$A_u$};
\node (c) at (2.5,0) {$A_u$};

\draw[font=\footnotesize,->]
	(a) edge node[above] {$\alpha$} (d)
	(a) edge node[left] {$\phi_u$} (b)
	(d) edge node[right] {$\phi_u$} (c)
	(b) edge node[above] {$\widetilde{\alpha}$} (c);

\end{tikzpicture}
\]
\end{lemma}

\begin{proof}
The composition $A\xrightarrow{\,\alpha\,}A\xrightarrow{\,\phi_u\,}A_u$ is a *-homomorphism, and the existence and uniqueness of $\widetilde{\alpha}\in\mathrm{Aut}(A_u)$ follows from the universal property. 
\end{proof}

\begin{prop}\label{prop:contact}
Let $A$ be a *-algebra, $(A_u,\phi_u)$ its universal C*-algebra (assume that it exists), $G$ a topological group, and $\alpha:G\to\mathrm{Aut}(A)$ a group action. Then,
there exists a unique action $\widetilde{\alpha}:G\to\mathrm{Aut}(A_u)$ such that $\phi_u\circ\alpha_g=\widetilde{\alpha}_g\circ\phi_u$. If, for every $a\in\phi_u(A)$, the map
\begin{equation}\label{eq:GmapA}
G\to A_u , \qquad
g\mapsto\widetilde\alpha_g(a) ,
\end{equation}
is continuous, then $\widetilde{\alpha}$ is strongly continuous.
\end{prop}

\begin{proof}
For all $g\in G$, the existence and uniqueness of $\widetilde{\alpha}_g$ follows from Lemma \ref{lemma:autA}.
Let $g_1,g_2\in G$ and denote by $e\in G$ the neutral element. From the commutative diagram
\[
\begin{tikzpicture}[>=To]

\node (a) at (0,1.8) {$A$};
\node (d) at (2.5,1.8) {$A$};
\node (b) at (0,0) {$A_u$};
\node (c) at (2.5,0) {$A_u$};

\draw[font=\footnotesize,->]
	(a) edge node[above] {$\alpha_e=\id_A$} (d)
	(a) edge node[left] {$\phi_u$} (b)
	(d) edge node[right] {$\phi_u$} (c)
	(b) edge[transform canvas={yshift=3pt}] node[above] {$\widetilde{\alpha}_e$} (c)
	(b) edge[transform canvas={yshift=-1pt}] node[below] {$\id_{A_u}$} (c);

\end{tikzpicture}
\]
and from Lemma \ref{lemma:autA} we deduce that $\widetilde{\alpha}_e=\id_{A_u}$. From the commutative diagram
\[
\begin{tikzpicture}[>=To,scale=1.2]

\node (a) at (0,1.8) {$A$};
\node (b) at (2.5,1.8) {$A$};
\node (c) at (5,1.8) {$A$};
\node (d) at (0,0) {$A_u$};
\node (e) at (2.5,0) {$A_u$};
\node (f) at (5,0) {$A_u$};

\draw[font=\footnotesize,->]
	(a) edge node[left] {$\phi_u$} (d)
	(b) edge node[fill=white] {$\phi_u$} (e)
	(c) edge node[right] {$\phi_u$} (f)
	(a) edge node[above] {$\alpha_{g_2}$} (b)
	(b) edge node[above] {$\alpha_{g_1}$} (c)
	(d) edge node[above] {$\widetilde\alpha_{g_2}$} (e)
	(e) edge node[above] {$\widetilde\alpha_{g_1}$} (f);

\end{tikzpicture}
\]
since the composition of the horizontal arrows on top is $\alpha_{g_1g_2}$, we get a commutative diagram
\[
\begin{tikzpicture}[>=To]

\node (a) at (0,1.8) {$A$};
\node (d) at (2.5,1.8) {$A$};
\node (b) at (0,0) {$A_u$};
\node (c) at (2.5,0) {$A_u$};

\draw[font=\footnotesize,->]
	(a) edge node[above] {$\alpha_{g_1g_2}$} (d)
	(a) edge node[left] {$\phi_u$} (b)
	(d) edge node[right] {$\phi_u$} (c)
	(b) edge[transform canvas={yshift=3pt}] node[above] {$\widetilde\alpha_{g_1}\widetilde\alpha_{g_2}$} (c)
	(b) edge[transform canvas={yshift=-1pt}] node[below] {$\widetilde\alpha_{g_1g_2}$} (c);

\end{tikzpicture}
\]
From Lemma \ref{lemma:autA} we deduce that $\widetilde\alpha_{g_1}\widetilde\alpha_{g_2}=\widetilde\alpha_{g_1g_2}$, i.e.~$\widetilde{\alpha}$ is a group action.

It remains to prove that, if the map \eqref{eq:GmapA} is continuous for every $a\in D:=\phi_u(A)$, then it is continuous for every $a\in A_u$. We use the density of $D$ in $A_u$ and a standard $\varepsilon/3$ argument.

Let $(g_\lambda)_{\lambda\in\Lambda}$ be a net in $G$ convergent to an element $g\in G$, call $T_\lambda:=\widetilde\alpha_{g_\lambda}$ and $T=\widetilde\alpha_{g}$.
Since *-homomorphisms between C*-algebras are contractive, one has $\|T_\lambda\|\leq 1$ and $\|T\|\leq 1$.

Let $a\in A_u$ be a fixed element (arbitrary) and $\varepsilon>0$. Due to the density of $D$, there exists a sequence $(a_k)_{k\in\N}$ in $D$ convergent to $a$. Thus, there exists $k_0$ such that, for every $\lambda\in\Lambda$ and $k\geq k_0$ one has
\[
\|T_\lambda a_k-T_\lambda a\|\leq \|a_k-a\|<\varepsilon/3.
\]
By hypothesis, $\lim_{\lambda\in\Lambda}T_\lambda a_k=Ta_k$ since $a_k\in D$.
Thus, for all $k\in\N$ there exists $\lambda_k$ such that, if $\lambda\geq \lambda_k$, then
\[
\|T_\lambda a_k-Ta_k\|<\varepsilon/3.
\]
Since $T$ is continuous, there exists $k_1$ such that, for all $k\geq k_1$, one has
\[
\|Ta_k-Ta\|<\varepsilon/3.
\]
If $k\geq\max\{k_0,k_1\}$ and $\lambda\geq \lambda_k$, then
\[
\|T_\lambda a-Ta\|\leq \|T_\lambda a-T_\lambda a_k\|+\|T_\lambda a_k-Ta_k\|+\|Ta_k-Ta\|<\varepsilon .
\]
Thus, $T_\lambda a\to Ta$. That is, for all $a\in A_u$, the map $G\to A$, $g\mapsto \widetilde{\alpha}_g(a)$, is continuous.
\end{proof}

\subsection{Graph C*-algebras}\label{sec:4}

In this section, we recall some general results from the theory of graph C*-algebras. Our main references are \cite{BPRS,R05}. We adopt the conventions of \cite{BPRS}, i.e.~the roles of source and range maps are exchanged with respect to~\cite{R05}.

Let $E=(E^0,E^1,s,t)$ be a directed graph, where $E^0$ is the set of vertices, $E^1$ is the set of edges, \mbox{$s:E^1\to E^0$} is the source map,
and \mbox{$t:E^1\to E^0$} is the target (or range) map. The graph is called \emph{row-finite} if $s^{-1}(v)$ is a finite set for every $v\in E^0$. It is called \emph{finite} if both sets $E^0$ and $E^1$ are finite. A \emph{sink} is a vertex $v$ with no outgoing edges, that is $s^{-1}(v)=\emptyset$.

\begin{df}[Leavitt path algebra]\label{def:Leavitt}
The \emph{Leavitt path algebra} $L_{\C}(E)$ of a row-finite graph $E$ is the universal algebra generated by a family
\begin{equation}\label{eq:CKfamily}
\big\{P_v,S_e,S_e^*:v\in E^0,e\in E^1\big\}
\end{equation}
satisfying the relations
\begin{alignat}{2}
P_vP_w &=\begin{cases}
P_v &\text{if }v=w \\
0 &\text{if }v\neq w
\end{cases}
 && \text{for all }v,w\in E^0,\tag{\text{CK0}} \label{eq:CK0} \\
S_e^*S_f &=\begin{cases}
P_{t(e)} &\text{if }e=f \\
0 &\text{if }e\neq f
\end{cases}\qquad
 && \text{for all }e\in E^1, \tag{\text{CK1}} \label{eq:CK1} \\
\sum_{e\in E^1:\,s(e)=v}\!\! S_eS_e^* &=P_v && \text{for all }v\in E^0\text{ that are not sinks},\tag{CK2} \label{eq:CK2} \\
P_{s(e)}S_e &=S_eP_{t(e)}=S_e && \text{for all }e\in E^1, \tag{CK3} \label{eq:CK3} \\
P_{t(e)}S_e^* &=S_e^*P_{s(e)}=S_e^* && \text{for all }e\in E^1. \tag{CK4} \label{eq:CK4}
\end{alignat}
\end{df}

A collection \eqref{eq:CKfamily} satisfying (\ref{eq:CK0}-\ref{eq:CK4}) is called a \emph{Cuntz--Krieger $E$-family}.
The algebra $L_{\C}(E)$ is universal in the sense that whenever $\{Q_v,T_e\}$ is another Cuntz--Krieger $E$-family in an algebra $A$, then there exists a unique *-homomorphism $L_{\C}(E)\to A$ that maps each $P_v$ to $Q_v$, $S_e$ to $T_e$ and $S_e^*$ to $T_e^*$.
We stress that the above universal property holds in the category of non-unital algebras, but
when $E$ is finite $L_{\C}(E)$ is unital with unit element $1=\sum_{v\in E^0}P_v$. It is also a *-algebra with involution defined by $P_v^*:=P_v$ and by declaring $S_e^*$ to be the adjoint of $S_e$ (as the notation suggests). Note that in a *-algebra
\eqref{eq:CK4} follows from \eqref{eq:CK3} by adjunction.

Leavitt path algebras can be defined over any field, but we will only consider the field of complex numbers.
The universal algebra of a Cuntz--Krieger $E$-family can be explicitly constructed in the obvious way, as a quotient of a free algebra. Similar definitions can be given for C*-algebras (with minimal modifications). In the case of C*-algebras the existence of the universal object is a non-trivial theorem \cite{R05}.

\begin{df}[Graph C*-algebra]
The \emph{graph \mbox{C*-algebra}} $C^*(E)$ of a row-finite graph $E$ is the universal \mbox{C*-algebra} generated by mutually 
orthogonal projections $\big\{P_v:v\in E^0\big\}$ and partial isometries $\big\{S_e:e\in E^1\big\}$ satisfying the \emph{Cuntz--Krieger relations}:
\begin{alignat}{2}
S_e^*S_e &=P_{t(e)} \qquad && \text{for all }e\in E^1\text{, and}\tag{\text{CK1'}} \label{eq:CK1p} \\
\sum_{e\in E^1:\,s(e)=v}\!\! S_eS_e^* &=P_v && \text{for all }v\in E^0\text{ that are not sinks.}\tag{CK2'} \label{eq:CK2p}
\end{alignat}
\end{df}

In $C^*(E)$, for all $v\in E^0$ and $e,f\in E^1$, one has
\begin{subequations}\label{eq:heavilyuse}
\begin{alignat}{2}
S_eP_v &=S_e \qquad\quad && \text{if }v=t(e) ,\\
S_e^*S_f &=0 && \text{if }e\neq f , \\
P_vS_e &=S_e && \text{if }v=s(e) .
\end{alignat}
\end{subequations}
The first relation follows from the definition of partial isometry and the fact that $P_v=S_e^*S_e$ is the source projection of $S_e$.
The second follows from the fact that the CK' relations imply that the range projections $S_eS_e^*$ are also mutually orthogonal (see \cite[page 309]{BPRS}). The third follows from the first two using \eqref{eq:CK2p}.

As a consequence, there is a *-homomorphism $L_{\C}(E)\to C^*(E)$ defined in the obvious way on Cuntz-Krieger $E$-families (with an abuse of notations, $P_v\mapsto P_v$ and $S_e\mapsto S_e$). One can prove that it is injective by using the
Graded Uniqueness Theorem, cf.~\cite{AAM17}.
From now on, we will think of $L_{\C}(E)$ as a dense *-subalgebra of $C^*(E)$ (dense since it contains all the generators).

\medskip

Any graph C*- algebra $C^*(E)$ can be endowed with a natural circle action 
\[
\alpha:U(1)\longrightarrow\mathrm{Aut}(C^*(E)),
\]
called the \emph{gauge action}, and given on generators by
$$
\alpha_u(P_v)=P_v\;,\qquad
\alpha_u(S_e)=u S_e\;,
$$
for all $u\in U(1)$, $v\in E^0$, $e\in E^1$.

We give now a slight reformulation of \cite[Theorem~2.1]{BPRS} that is more suitable for the purposes of this work
(see also \cite[Theorem~2.2]{R05}).

\begin{thm}[Gauge-invariant Uniqueness Theorem]\label{thm:gut}
Let $E$ be a row-finite graph with Cuntz--Krieger family $\{S,P\}$, 
let $A$ be a C*-algebra with a continuous action of $U(1)$ and $\rho:C^*(E)\to A$ a $U(1)$-equivariant 
*-homomorphism. If $\rho(P_v)\neq 0$ for all $v\in E^0$, then $\rho$ is injective.
\end{thm}

\section{The coordinate algebra of a quantum sphere}\label{sec:3}
Let $n\in\N$ and $0\leq q<1$. We denote by $\mathcal{A}(S^{2n+1}_q)$ the unital (associative, complex) *-algebra generated by elements $\{z_i,z_i^*\}_{i=0}^n$ with relations
\begin{subequations}\label{eq:qsphere}
\begin{align}
z_jz_i &=qz_iz_j &&\forall\;0\leq i<j\leq n \;, \label{eq:qsphereA} \\
z_i^*z_j &=qz_jz_i^* &&\forall\;0\leq i\neq j\leq n \;, \label{eq:qsphereB} \\
z_i^*z_i &=z_iz_i^*+(1-q^2)\sum\nolimits_{j=i+1}^n z_jz_j^* 
    &&\forall\;0\leq i\leq n \;,\label{eq:qsphereC} \\
z_0z_0^*+z_1z_1^* &+\ldots+z_nz_n^*=1 \;. \label{eq:qsphereE}
\end{align}
\end{subequations}
Here we adopt the original notations of \cite{VS91}, which are also the notations adopted in \cite{HS02}.

Observe that $z_n$ is normal and that $\mathcal{A}(S^{1}_q)\cong\mathcal{A}(S^{1})$ is generated by a single unitary operator, that we can identify with the identity function on the unit circle $S^1\subseteq\C$.

The relations for $q=0$, written explicitly are
\begin{subequations}\label{eq:qspherep}
\begin{align}
z_iz_j &=0 &&\forall\;i>j \;, \label{eq:qsphereAp} \\
z_i^*z_j &=0 &&\forall\;i\neq j \;, \label{eq:qsphereBp} \\
z_i^*z_i &=\sum\nolimits_{j\geq i} z_jz_j^*  \;,\label{eq:qsphereCp} \\
\sum\nolimits_jz_jz_j^* & =1 \;. \label{eq:qsphereEp}
\end{align}
\end{subequations}

Recall that an algebra $A$ is called \emph{$G$-graded}, where $G$ is a group, if it has a decomposition $A=\bigoplus_{g\in G}A_g$ into vector subspaces $\{A_g:g\in G\}$ satisfying $A_g\cdot A_{g'}\subseteq A_{gg'}$ for all $g,g'\in G$. We call $A_g$ a \emph{spectral subspace} of $A$, and we say that its elements are \emph{homogeneous of degree $g$}.

\medskip

An action $\alpha$ of $\mv{t}=(t_0,\ldots,t_n)\in\mathbb{T}^{n+1}$ on $\mathcal{A}(S^{2n+1}_q)$ is defined on generators by
\begin{equation}\label{eq:Taction}
\alpha_{\mv{t}}(z_i):=t_iz_i ,\qquad\forall \;0\leq i\leq n .
\end{equation}
The action \eqref{eq:Taction} defines a $\Z^{n+1}$-grading in the usual way:
an element $a\in\mathcal{A}(S^{2n+1}_q)$ has degree $\mv{m}=(m_0,\ldots,m_n)\in\Z^{n+1}$ if
\begin{equation}\label{eq:spectralsubspace}
\alpha_{\mv{t}}(a)=t_0^{m_0}\ldots t_n^{m_n}a \qquad\text{for all }\mv{t}\in\mathbb{T}^{n+1}.
\end{equation}
The composition of \eqref{eq:Taction} with the diagonal embedding
\[
U(1)\to\mathbb{T}^{n+1} , \qquad u\mapsto (u,\ldots,u) ,
\]
gives an action of $U(1)$ on $\mathcal{A}(S^{2n+1}_q)$.
For the sake of completeness, we mention that the $\Z$-grading associated to this action is a strong grading, as proved for example in a more general context in \cite{DL15} (see also \cite[Sect.~5.1]{ADL15}). This means that $\mathcal{A}(S^{2n+1}_q)$ is a Hopf-Galois extension of its degree $0$ part, cf.~e.g.~\cite[Theorem 8.1.7]{Mon93}. The subalgebra of $U(1)$-invariant elements is denoted 
$\mathcal{A}(\C P^n_q)$ and the geometric picture is that of a principal $U(1)$-bundle $S^{2n+1}_q\to\C P^n_q$ with total space a quantum sphere and base space a quantum projective space (see e.g.~\cite{ADL15,AKL14} and references therein).

The relevant subalgebra for us is the one of $\mathbb{T}^{n+1}$-invariant elements, i.e.~elements of degree $\mv{m}=(0,\ldots,0)$. This is strictly smaller than $\mathcal{A}(\C P^n_q)$ and will be studied in the next subsection.

\subsection{Algebraic properties}\label{sec:alg}
Here we collect some preliminary results about the algebra $\mathcal{A}(S^{2n+1}_q)$ that we need later on to study the associated universal C*-algebra.

\begin{prop}\label{prop:Zproj}
Let $q=0$. Then $z_i$ is a partial isometry, for all $0\leq i\leq n$, and
$$
\bigl\{z_iz_i^*:0\leq i\leq n\bigr\}
$$
is a set of mutually orthogonal projections with sum $1$.
\end{prop}

\begin{proof}
Using \eqref{eq:qsphereEp}, the relation \eqref{eq:qsphereCp} can be rewritten as
\begin{equation}\label{eq:qsphereDp}
z_i^*z_i =1-\sum\nolimits_{j<i} z_jz_j^* .
\end{equation}
From \eqref{eq:qsphereDp} and \eqref{eq:qsphereAp} it follows that
\[
z_iz_i^*z_i=z_i-\sum\nolimits_{j<i} \equalto{(z_iz_j)}{0}z_j^*=z_i ,
\]
thus proving the first claim. Since $z_i$ is a partial isometry, $z_iz_i^*$ is a projection, and orthogonality follows
from \eqref{eq:qsphereBp}. From \eqref{eq:qsphereEp} it follows that the sum of these projections is $1$.
\end{proof}

\begin{lemma}\label{lemma:commut}
Let $0\leq q<1$. Then,
\begin{enumerate}
\item\label{en:commutA} For every $0\leq i<j\leq n$, $z_iz_i^*$ and $z_j$ commute.

\item\label{en:commutB} For every $0\leq i\leq n$ and every positive integer $m$, one has
\begin{equation}\label{eq:relk}
z_i^*(z_i)^m=q^{2m}(z_i)^{m}z_i^*+(1-q^{2m})(z_i)^{m-1}\left(1-\sum\nolimits_{j=0}^{i-1}z_jz_j^*\right) .
\tag{rel$_m$}
\end{equation}
\end{enumerate}
\end{lemma}

\begin{proof}
\ref{en:commutA} This is a simple computation using \eqref{eq:qsphereA} and \eqref{eq:qsphereB}.

\medskip

\noindent
\ref{en:commutB} This is proved by induction on $m\geq 1$.
From \eqref{eq:qsphereC} and \eqref{eq:qsphereE} we derive (rel$_1$), that is
\begin{align*}
z_i^*z_i &=q^{2}z_iz_i^*+(1-q^{2})\sum\nolimits_{j=i}^nz_jz_j^*
&& \text{by \eqref{eq:qsphereC}} \\
 &=q^{2}z_iz_i^*+(1-q^{2})\left(1-\sum\nolimits_{j=0}^{i-1}z_jz_j^*\right) .
&& \text{by \eqref{eq:qsphereE}} 
\end{align*}
Now, let $m\geq 2$ and assume, by inductive hypothesis, that (rel$_{m-1}$) holds. Then,
\begin{align*}
z_i^* &(z_i)^m  =\big(z_i^*(z_i)^{m-1}\big)z_i \\
&=q^{2m-2}(z_i)^{m-1}z_i^*z_i+(1-q^{2m-2})(z_i)^{m-2}\left(1-\sum\nolimits_{j=0}^{i-1}z_jz_j^*\right)z_i
&& \text{by (rel$_{m-1}$)}\\
&=q^{2m-2}(z_i)^{m-1}z_i^*z_i+(1-q^{2m-2})(z_i)^{m-1}\left(1-\sum\nolimits_{j=0}^{i-1}z_jz_j^*\right)
&& \text{by point \ref{en:commutA}}\\
&=q^{2m}(z_i)^mz_i^*+\Big(q^{2m-2}(1-q^{2})+(1-q^{2m-2})\Big)
(z_i)^{m-1}\left(1-\sum\nolimits_{j=0}^{i-1}z_jz_j^*\right)
&& \text{by (rel$_1$)} \\
&=q^{2m}(z_i)^{m}z_i^*
+(1-q^{2m})(z_i)^{m-1}\left(1-\sum\nolimits_{j=0}^{i-1}z_jz_j^*\right)
.
\end{align*}
The last line is exactly the right hand side of \eqref{eq:relk}.
\end{proof}

Let us stress that Lemma \ref{lemma:commut}\ref{en:commutB} holds also for $q=0$ (although we will only need it for $q\neq 0$), and simplifies to $z_i^*(z_i)^m=(z_i)^{m-1}$.

\begin{prop}\label{lemma:genfam}
Let $0\leq q<1$. For all $\underline{i}=(i_0,\ldots,i_n)\in\N^n\times\Z$ and $\underline{j}=(j_0,\ldots,j_{n-1})\in\N^n$ define
\[
\inner{\underline{i},\underline{j}}:=z_0^{i_0}\ldots z_{n-1}^{i_{n-1}} z_n^{i_n}  \big( z_0^{j_0}\ldots z_{n-1}^{j_{n-1}} \big)^*
\]
if $i_n\geq 0$ and
\[
\inner{\underline{i},\underline{j}}:=z_0^{i_0}\ldots z_{n-1}^{i_{n-1}} (z_n^*)^{-i_n}  \big( z_0^{j_0}\ldots z_{n-1}^{j_{n-1}} \big)^*
\]
if $i_n<0$. Then, the set
\begin{equation}\label{eq:basis}
\big\{ \inner{\underline{i},\underline{j}} \;:\;
\underline{i}\in\N^n\times\Z,\underline{j}\in\N^n
\big\}
\end{equation}
is a generating family of the vector space $\mathcal{A}(S^{2n+1}_q)$.
\end{prop}

\begin{proof}
Let $J$ be the span of the vectors in the set \eqref{eq:basis}. Since $J$ contains the unit, it is enough to show that it is an ideal in $\mathcal{A}(S^{2n+1}_q)$, i.e.~that it is closed under involution and under left multiplication by the elements $z_i,z_i^*$ for all $0\leq i\leq n$ (automatically $J$ will be closed under right multiplication by the generators). The first claim is obvious, since the adjoint of
$\inner{\underline{i},\underline{j}}$ is, by construction, $\inner{\underline{j},-i_n,i_0,\ldots,i_{n-1}}$. Concerning the second claim, let us start with $q\neq 0$. Let $0\leq m\leq n$. If $m<n$, or $m=n$ and $i_n\geq 0$, using \eqref{eq:qsphereA} we get:
\[
z_m\inner{\underline{i},\underline{j}}=q^{i_0+\ldots+i_{m-1}}\inner{i_0,\ldots,i_m+1,\ldots,i_n,\underline{j}} ,
\]
which in particular implies that $z_mJ\subseteq J$ for $m<n$.
If $m=n$ and $i_n<0$, using \eqref{eq:qsphereA}:
\begin{align*}
z_n\inner{\underline{i},\underline{j}} &=q^{i_0+\ldots+i_{n-1}}\inner{i_0,\ldots,i_{n-1},0,\ldots,0}z_nz_n^*\inner{0,\ldots,0,i_n+1,\underline{j}} \\
\intertext{then using \eqref{eq:qsphereE}}
&=q^{i_0+\ldots+i_{n-1}}\inner{i_0,\ldots,i_{n-1},i_n+1,\underline{j}} \\
&\qquad -q^{i_0+\ldots+i_{n-1}}\sum_{k=0}^{n-1}\inner{i_0,\ldots,i_{n-1},0,\ldots,0}z_kz_k^*\inner{0,\ldots,0,i_n+1,\underline{j}} \\
\intertext{and \eqref{eq:qsphereA} (and adjoint) again}
&=q^{i_0+\ldots+i_{n-1}}\inner{i_0,\ldots,i_{n-1},i_n+1,\underline{j}} \\
&\qquad -\sum_{k=0}^{n-1}
q^{i_0+\ldots+i_k+3(i_{k+1}+\ldots +i_{n-1})+(|i_n|-1)}
\inner{\ldots,i_k+1,\ldots,i_n+1,\ldots,j_k+1,\ldots}.
\end{align*}
Thus, $z_nJ\subseteq J$ as well.

Next, if $i_n<0$, using \eqref{eq:qsphereB} we find:
\[
z_n^*\inner{\underline{i},\underline{j}}=q^{i_0+\ldots+i_{n-1}}\inner{i_0,\ldots,i_{n-1},i_n-1,\underline{j}} .
\]
Finally, for $m<n$ or $m=n$ and $i_n\geq 0$,
we use Lemma \ref{lemma:commut}\ref{en:commutB}, in the form
\begin{equation}\label{eq:lemmaintheform}
z_m^*(z_m)^{i_m}=q^{2i_m}(z_m)^{i_m}z_m^*+(1-q^{2i_m})(z_m)^{i_m-1}\left(1-\sum\nolimits_{l=0}^{m-1}z_lz_l^*\right) .
\end{equation}
We compute
\begin{align*}
z_m^*\inner{\underline{i},\underline{j}} 
&=z_m^*\inner{i_0,\ldots,i_{m-1},0,\ldots,0}(z_m)^{i_m}
\inner{0,\ldots,0,i_{m+1},\ldots,i_n,\underline{j}}
\\
\intertext{which using \eqref{eq:qsphereB} becomes}
&=q^{i_0+\ldots+i_{m-1}}\inner{i_0,\ldots,i_{m-1},0,\ldots,0}z_m^*(z_m)^{i_m}
\inner{0,\ldots,0,i_{m+1},\ldots,i_n,\underline{j}}
\\
\intertext{and using \eqref{eq:lemmaintheform} becomes}
&=q^{i_0+\ldots+i_{m-1}+2i_m}\inner{\underline{i},j_0,\ldots,j_m+1,\ldots,j_{n-1}} \\
+q^{i_0+\ldots+i_{m-1}}&(1-q^{2i_m})\inner{i_0,\ldots,i_m-1,0,\ldots,0}\left(1-\sum\nolimits_{l=0}^{m-1}z_lz_l^*\right)
\inner{0,\ldots,0,i_{m+1},\ldots,i_n,\underline{j}} .
\end{align*}
Since on the right hand side only generators $z_l,z_l^*$ with $l<m$ appear,
using the last relation one can prove that $z_m^*J\subseteq J$ by induction on $m$. If $z_l^*J\subseteq J$ for all $l<m$, from the last equality it follows that
\[
z_m^*J\subseteq J+z_0^{i_0}\ldots z_m^{i_{m}-1}\sum_{l=0}^{m-1}z_lJ .
\]
But we already proved that $z_iJ\subseteq J$ for all $i$, hence the claim.
It remains to prove the starting point of the induction, i.e.~the case $m=0$. This is easy. For $m=0$, using again \eqref{eq:lemmaintheform},
we get
\[
z_0^*\inner{\underline{i},\underline{j}}=q^{2i_0+i_1+\ldots+i_n}\inner{\underline{i},j_0+1,\ldots,j_{n-1}}+(1-q^{2i_0})\inner{i_0-1,\ldots,i_n,\underline{j}} ,
\]
thus completing the proof for $q\neq 0$. The computation can be repeated for $q=0$, with powers $q^k$ replaced by $\delta_{k,0}$ (observe that negative powers never appear in the proof).
\end{proof}

Using the Diamond Lemma like in \cite[Sect.~4.1.5]{KS97}, it should be possible to prove that the set \eqref{eq:basis} is a basis of the vector space $\mathcal{A}(S^{2n+1}_q)$. We shall not prove this statement since a generating family is enough for our purposes.

\smallskip

We denote by $\bigr\{\ket{\underline{k}}:\underline{k}=(k_0,\ldots,k_n)\in\N^n\times\Z\bigr\}$ the standard basis of $\ell^2(\N^n\times\Z)$. A strongly continuous unitary representation $U$ of $\mathbb{T}^{n+1}$ on $\ell^2(\N^n\times\Z)$ is given in the standard basis by
\begin{equation}\label{eq:unitary}
U_{\mv{t}}\ket{k_0,\ldots,k_{n-1},k_n}=t_0^{k_0}t_1^{k_1}\cdots t_n^{k_n} \ket{k_0,\ldots,k_{n-1},k_n}
\end{equation}
for all $\mv{t}=(t_0,\ldots,t_n)\in\mathbb{T}^{n+1}$.

\begin{prop}\label{prop:rep}
A bounded *-representation $\pi$ of $\mathcal{A}(S^{2n+1}_q)$ on $\ell^2(\N^n\times\Z)$ is defined on generators as follows.
If $q\neq 0$,
\begin{align*}
\pi(z_0)\ket{k_0,\ldots,k_n} &=\sqrt{1-q^{2(k_0+1)}}\ket{k_0+1,k_1,\ldots,k_n} , \\
\pi(z_i)\ket{k_0,\ldots,k_n} &=q^{k_0+\ldots+k_{i-1}}\sqrt{1-q^{2(k_i+1)}}\ket{k_0,\ldots,k_i+1,\ldots,k_n}  && \text{for }0<i<n, \\
\pi(z_n)\ket{k_0,\ldots,k_n} &=q^{k_0+\ldots+k_{n-1}}\ket{k_0,\ldots,k_{n-1},k_n+1} .
\end{align*}
If $q=0$, then $\pi(z_i)=:Z_i$ are given by the formulas
\begin{align*}
Z_0\ket{k_0,\ldots,k_n} &=\ket{k_0+1,k_1,\ldots,k_n} , \\
Z_i\ket{k_0,\ldots,k_n} &=\delta_{k_0,0}\ldots\delta_{k_{i-1},0}\ket{0,\ldots,0,k_i+1,k_{i+1},\ldots,k_n}  && \text{for }0<i<n, \\
Z_n\ket{k_0,\ldots,k_n} &=\delta_{k_0,0}\ldots\delta_{k_{n-1},0}\ket{0,\ldots,0,k_n+1} .
\end{align*}
This representation is $\mathbb{T}^{n+1}$-covariant with respect to the action \eqref{eq:Taction} and the representation \eqref{eq:unitary}.
\end{prop}

\begin{proof}
It is a straightforward to check that the relations \eqref{eq:qsphere} are satisfied, and since $U_{\mv{t}}\pi(z_i)U_{\mv{t}}^*=t_i\pi(z_i)=\pi(\alpha_{\mv{t}}(z_i))$, the representation $\pi$ is covariant (for all $0\leq q<1$).
\end{proof}

\begin{prop}\label{prop:joint}
Let $\pi$ be the representation in Prop.~\ref{prop:rep}.
\begin{enumerate}
\item\label{en:jointA}
The elements
\begin{equation}\label{eq:mutuallycommuting}
\big\{z_iz_i^*:0\leq i\leq n\big\} .
\end{equation}
generate a commutative subalgebra of $\mathcal{A}(S^{2n+1}_q)$ that we denote by
$\mathcal{A}(\Delta^n_q)$.

\item\label{en:jointB}
For $q\neq 0$, the set of joint eigenvalues of $\big(\pi(z_0z_0^*),\ldots,\pi(z_nz_n^*)\big)$ is given by all tuples
of the form
\[
\big(1-q^{2k_0},(1-q^{2k_1})q^{2k_0},\ldots,(1-q^{2k_{n-1}})q^{2(k_0+\ldots+k_{n-2})},q^{2(k_0+\ldots+k_{n-1})}\big)
\]
with $k_0,\ldots,k_{n-1}\in\N$.

\item\label{en:jointC}
For $q=0$, the set of joint eigenvalues is the set of corner points of the standard $n$-simplex $\Delta^n$.
\end{enumerate}
\end{prop}

\begin{proof}
The first statement immediately follows from Lemma \ref{lemma:commut}\ref{en:commutA}.
The second statement is proved by a direct computation, using the fact that all operators are diagonal, that is:
\begin{subequations}\label{eq:diagonal}
\begin{align}
\pi(z_iz_i^*)\ket{k_0,\ldots,k_{n-1},k_n} &=q^{2(k_0+\ldots+k_{i-1})}(1-q^{2k_i})\ket{k_0,\ldots,k_{n-1},k_n} \\
\intertext{if $0\leq i<n$ and}
\pi(z_nz_n^*)\ket{k_0,\ldots,k_{n-1},k_n} &=q^{2(k_0+\ldots+k_{n-1})}\ket{k_0,\ldots,k_{n-1},k_n} .
\end{align}
\end{subequations}

For $q=0$, from Prop.~\ref{prop:Zproj} and the fact that the eigenvalue of a projection is $0$ or $1$, we deduce that that joint eigenvalues must be vectors in the standard basis of $\R^{n+1}$, i.e.~the corner points of the standard $n$-simplex. For all $0\leq i,j\leq n$ one has
\[
\pi(z_jz_j^*)\bigl|0,\ldots,0,\stackrel{\substack{i \\ \downarrow}}{1},*,\ldots,*\bigr>
=\delta_{i,j}\bigl|0,\ldots,0,\stackrel{\substack{i \\ \downarrow}}{1},*,\ldots,*\bigr> ,
\]
which proves that all tuples of the standard basis are joint eigenvalues.
\end{proof}

\begin{rem}\label{rem:qsimp}
The closure in $\R^{n+1}$ of the set in Prop.~\ref{prop:joint}\ref{en:jointB} is called the \emph{quantized $n$-simplex} in \cite{MK22}, and we shall denote it by $\Delta^n_q$.
\end{rem}

Let $\vartheta:\mathcal{A}(S^{2n+1}_q)\to\mathcal{A}(S^{2n+1}_q)^{\mathbb{T}^{n+1}}$ be the map
\begin{equation}\label{eq:maptheta}
\vartheta(a):=\int_{\mathbb{T}^{n+1}}\alpha_{\mv{t}}(a)d\mu_t ,
\end{equation}
where $a\in\mathcal{A}(S^{2n+1}_q)$ and $\mu$ is the normalized Haar measure on $\mathbb{T}^{n+1}$.

We now show that, for $q\neq 0$, the representation $\pi$ is faithful.
The proof for $q=0$ will be given in the Sect.~\ref{sec:5} using graph C*-algebras.

\begin{prop}\label{prop:theta}
Let $q\neq 0$. Then,
\begin{enumerate}
\item\label{en:thetaA} the map $\vartheta$ in \eqref{eq:maptheta} is the identity on $\mathcal{A}(\Delta^n_q)$;

\item\label{en:thetaB} $\mathcal{A}(S^{2n+1}_q)^{\mathbb{T}^{n+1}}=\mathcal{A}(\Delta^n_q)$.
\end{enumerate}
\end{prop}

\begin{proof}
By construction, $\vartheta$ is the identity on $\mathbb{T}^{n+1}$-invariant elements. Clearly
\begin{equation}\label{eq:oino}
\mathcal{A}(\Delta^n_q)\subseteq\mathcal{A}(S^{2n+1}_q)^{\mathbb{T}^{n+1}} ,
\end{equation}
since $\mathcal{A}(\Delta^n_q)$ is generated by invariant elements, hence \ref{en:thetaA}. If we show that the image of $\vartheta$ is in $\mathcal{A}(\Delta^n_q)$, this proves the inclusion opposite to \eqref{eq:oino}, and then \ref{en:thetaB}.

It is enough to show that $\vartheta(\inner{\underline{i},\underline{j}})\in\mathcal{A}(\Delta^n_q)$ for all elements in the set \eqref{eq:basis}. For $\underline{t}\in\mathbb{T}^{n+1}$,
\begin{equation}\label{eq:actmon}
\alpha_{\mv{t}}(\inner{\underline{i},\underline{j}})=t_0^{i_0-j_0}\ldots t_{n-1}^{i_{n-1}-j_{n-1}}t_n^{i_n}\inner{\underline{i},\underline{j}} ,
\end{equation}
and the integral is zero unless $i_k=j_k$ for all $0\leq k<n$ and $i_n=0$.
Thus $\vartheta(\inner{\underline{i},\underline{j}})$ is either zero or of the form
\begin{equation}\label{eq:monomialform} 
z_0^{j_0}\ldots z_{n-1}^{j_{n-1}}\big( z_0^{j_0}\ldots z_{n-1}^{j_{n-1}} \big)^* .
\end{equation}
We must show that these monomials belong to $\mathcal{A}(\Delta^n_q)$.
Using \eqref{eq:qsphereA} and \eqref{eq:qsphereB} we rewrite \eqref{eq:monomialform} in the form
\[
q^{(\ldots)}z_0^{j_0}(z_0^*)^{j_0}\ldots z_{n-1}^{j_{n-1}}(z_{n-1}^*)^{j_{n-1}}
\]
where the power of $q$ is irrelevant. It remains to show that, for all $0\leq i\leq n-1$ and all $m\in\N$, the monomial
$z_i^m(z_i^*)^m$ belongs to $\mathcal{A}(\Delta^n_q)$. We prove it by induction on $m$. It is trivially true if $m=0$.
Assume that it is true for a fixed $m\geq 0$. We have
\begin{align*}
z_i^{m+1} &(z_i^*)^{m+1}=z_i(z_i^mz_i^*)(z_i^*)^{m} \\
&=z_i\left(
q^{-2m}z_i^*z_i^m+(1-q^{-2m})(z_i)^{m-1}\left(1-\sum\nolimits_{j=0}^{i-1}z_jz_j^*\right)
\right)(z_i^*)^{m} && \text{by \eqref{eq:relk}} \\
&=\left(
q^{-2m}z_iz_i^*+(1-q^{-2m})\left(1-\sum\nolimits_{j=0}^{i-1}z_jz_j^*\right)
\right)z_i^m(z_i^*)^m && \text{by Lemma \ref{lemma:commut}\ref{en:commutA}}.
\end{align*}
In the last line, the expression inside the parentheses belongs to $\mathcal{A}(\Delta^n_q)$, $z_i^m(z_i^*)^m\in\mathcal{A}(\Delta^n_q)$ by inductive hypothesis, and then $z_i^{m+1}(z_i^*)^{m+1}\in\mathcal{A}(\Delta^n_q)$ as well.
\end{proof}

\begin{rem}\label{rem28}
Prop.~\ref{prop:theta}\ref{en:thetaB} fails for $q=0$ (and $n\geq 1$). 
To see this, notice that
\begin{equation}\label{eq:projm}
Z_0^m(Z_0^*)^m-Z_0^{m+1}(Z_0^*)^{m+1}
\end{equation}
is a projection onto the subspace of $\ell^2(\N^n\times\Z)$ spanned by vectors $\ket{\underline{k}}$ with $k_0=m$. 
Being mutually orthogonal projections, the operators in \eqref{eq:projm} are linearly independent, and so are the corresponding elements $z_0^m(z_0^*)^m-z_0^{m+1}(z_0^*)^{m+1}\in\mathcal{A}(S^{2n+1}_0)^{\mathbb{T}^{n+1}}$. Thus, $\mathcal{A}(S^{2n+1}_0)^{\mathbb{T}^{n+1}}$ is infinite-dimensional. On the other hand, $\mathcal{A}(\Delta^n_0)$ is finite-dimensional, being the linear span of the $n+1$ projections in Prop.~\ref{prop:Zproj}. Similarly,
the next Lemma \ref{lemma:faithrep}\ref{lemma:faithrepB} fails when $q=0$, since $\mathcal{A}(\Delta^n_0)$ is finite-dimensional and $\C[x_1,\ldots,x_n]$ is infinite-dimensional.
\end{rem}

\begin{lemma}\label{lemma:faithrep}
Let $q\neq 0$. Then:
\begin{enumerate}
\item\label{lemma:faithrepA}
the restriction $\pi:\mathcal{A}(\Delta^n_q)\to\mathcal{B}(\ell^2(\N^n\times\Z))$ of the representation
in Prop.~\ref{prop:rep} is faithful;

\item\label{lemma:faithrepB}
$\mathcal{A}(\Delta^n_q)$ is isomorphic to the *-algebra $\C[x_1,\ldots,x_n]$ of polynomials in $n$ real indeterminates.
\end{enumerate}
\end{lemma}

\begin{proof}
A *-homomorphism $f:\C[x_1,\ldots,x_n]\to\mathcal{A}(\Delta^n_q)$ is defined on generators by
\[
f(x_i):=\sum_{j=i}^nz_jz_j^* ,\qquad\text{ for all }1\leq i\leq n .
\]
It is surjective, since $z_0z_0^*=1-f(x_1)$, $z_iz_i^*=f(x_i)-f(x_{i+1})$ for
$0<i<n$, and $z_nz_n^*=f(x_n)$. We now show that the composition $\pi\circ f$ is injective,
and hence $f$ is injective and $\pi$ is injective on the image of $f$, thus concluding the proof of both points \ref{lemma:faithrepA} and \ref{lemma:faithrepB}.

If $P(x_1,\ldots,x_n)$ is a polynomial in the kernel of $\pi\circ f$, by definition of the representation it means that
\[
P(q^{2k_0},q^{2(k_0+k_1)},\ldots,q^{2(k_0+\ldots+k_{n-1})})=0 \qquad\forall\;k_0,\ldots,k_{n-1}\in\N.
\]
Thus, $P$ vanishes on the set $S_n$ of tuples of the form $(q^{m_1},\ldots,q^{m_n})$, with $m_1,\ldots,m_n\in\N$ and $0\leq m_1\leq m_2\leq\ldots\leq m_n$. We now prove by induction on $n\geq 1$ that, if a polynomial in $n$ variables vanishes on $S_n$, then it must be zero. For $n=1$, this follows from the fact that a non-zero polynomial in one variable has only finitely many zeros. For arbitrary $n\geq 2$, we can write such a polynomial $P$ as
\[
P(x_1,\ldots,x_n)=\sum_{N\geq 0}P_N(x_1,\ldots,x_{n-1}) x_n^N ,
\]
where $P_N\in\C[x_1,\ldots,x_{n-1}]$.
For every fixed $(\lambda_1,\ldots,\lambda_{n-1})\in S_{n-1}$,
\[
P(\lambda_1,\ldots,\lambda_{n-1},y)=\sum_{N\geq 0}P_N(\lambda_1,\ldots,\lambda_{n-1}) y^N .
\]
is a polynomial in $y$ with infinitely many zeroes. Hence it must be zero, which means that all coefficients 
$P_N(\lambda_1,\ldots,\lambda_{n-1})$ are zero. Thus, the polynomials $P_N$ all vanish on $S_{n-1}$.
By inductive hypothesis, $P_N=0$ for all $N$, and so $P=0$.
\end{proof}

\begin{prop}\label{prop:faithrep}
If $q\neq 0$,
the representation $\pi:\mathcal{A}(S^{2n+1}_q)\to\mathcal{B}(\ell^2(\N^n\times\Z))$ of Prop.~\ref{prop:rep} is faithful.
\end{prop}

\begin{proof}
Since the vector space $\mathcal{A}(S^{2n+1}_q)$ is a direct sum of spectral subspaces of the action of $\mathbb{T}^{n+1}$, it is enough to show that $\pi$ is injective on each of these subspaces.

Let $a\in\mathcal{A}(S^{2n+1}_q)$ be an element in the kernel of $\pi$ homogeneous of degree $\mv{m}\in\Z^{n+1}$, i.e.~satisfying \eqref{eq:spectralsubspace}.
We now show that $a=0$, thus proving injectivity.
Write
\[
a=\sum_{\underline{i}\in\N^n\times\Z,\underline{j}\in\N^n:\underline{i}-(\underline{j},0)=\underline{m}}\lambda_{\underline{i},\underline{j}}\inner{\underline{i},\underline{j}}
\]
as a linear combination of the elements in \eqref{eq:basis}, where the constrain $\underline{i}-(\underline{j},0)=\underline{m}$ comes from the condition \eqref{eq:spectralsubspace}.
Observe that we can reshuffle the $z$'s in $\inner{\underline{i},\underline{j}}$
and put them in an arbitrary order, getting a power of $q$ in front. We can reshuffle the $z^*$'s as well (as long as
we don't move $z_i$ to the right of $z_i^*$, there is no problem).
Each $\inner{\underline{i},\underline{j}}$ can then be written in the form
\[
l_0l_1\ldots l_nc_{\underline{i},\underline{j}}r_n\ldots r_1r_0
\]
where, for every $0\leq k\leq n$, we set $l_k:=z_k^{m_k}$ and $r_k:=1$ if $m_k\geq 0$, and we set
$l_k:=1$ and $r_k:=(z_k^*)^{-m_k}$ if $m_k<0$. By construction, $c_{\underline{i},\underline{j}}$ is $\mathbb{T}^{n+1}$-invariant (and is the only piece depending on the labels $\underline{i},\underline{j}$). Adding up all these elements, we get
\[
a=l_0l_1\ldots l_nb\,r_n\ldots r_1r_0
\]
where $b$ is some $\mathbb{T}^{n+1}$-invariant element.
By Prop.~\ref{prop:theta}\ref{en:thetaB}, $b$ belongs to $\mathcal{A}(\Delta^n_q)$.
Now, $\pi(a)=0$ implies that $\pi(b)$ maps the range of $\pi(r_n\ldots r_1r_0)$ to the kernel of $\pi(l_0l_1\ldots l_n)$.
But each $\pi(z^*_i)$ has full range, and each $\pi(z_i)$ has trivial kernel ($\pi(z_i^*z_i)$ is diagonal with only non-zero eigenvalues). Hence, it must be $\pi(b)=0$. From Lemma \ref{lemma:faithrep}\ref{lemma:faithrepA} we deduce that $b=0$, and then $a=0$ as well.
\end{proof}

\subsection{The graph C*-algebra of a quantum sphere}\label{sec:5}

We denote by $\Sigma_n$ the graph with
$n+1$ vertices $\{v_0,v_1,\ldots,v_n\}$
and one edge $e_{i,j}$ from $v_i$ to $v_j$ for all $0\leq i\leq j\leq n$.
A picture is in Figure~\ref{fig:sphere}.
We denote by $\widetilde{\Sigma}_n$ the graph with
$n+1$ vertices $\{w_0,w_1,\ldots,w_n\}$,
one edge $f_i$ from $w_i$ to $w_{i+1}$ for all $0\leq i<n$,
and one loop $e_i$ at $w_i$ for all $0\leq i\leq n$.
A picture is in Figure~\ref{fig:sphereB}.
With an abuse of notations, we denote by $\{P,S\}$ the Cuntz--Krieger family of both graphs $\Sigma_n$ and $\widetilde{\Sigma}_n$.

\begin{figure}[t]
\begin{center}
\begin{tikzpicture}[>=stealth,node distance=2cm,
main node/.style={circle,inner sep=2pt},
freccia/.style={->,shorten >=2pt,shorten <=2pt},
ciclo/.style={out=130, in=50, loop, distance=2cm}]

\clip (-0.6,-2.7) rectangle (10.6,1.4);

\node[main node] (1) {};
\node (2) [main node,right of=1] {};
\node (3) [main node,right of=2] {};
\node (4) [main node,right of=3] {};
\node (5) [right of=4] {};
\node (6) [main node,right of=5] {};

\filldraw (1) circle (0.06) node[below left] {$v_0$};
\filldraw (2) circle (0.06);
\filldraw (3) circle (0.06);
\filldraw (4) circle (0.06);
\filldraw (6) circle (0.06) node[below right] {$v_n$};

\path[freccia] (1) edge[ciclo] (1);
\path[freccia] (2) edge[ciclo] (2);
\path[freccia] (3) edge[ciclo] (3);
\path[freccia] (4) edge[ciclo] (4);
\path[freccia] (6) edge[ciclo] (6);

\path[freccia]
	(1) edge (2)
	(2) edge (3)
	(3) edge (4)
	(4) edge[dashed]  (6)
	(1) edge[bend right] (3)
	(1) edge[bend right=40] (4)
	(2) edge[bend right] (4);

\path[white]
	(1) edge[bend right=60] coordinate (7) (6)
	(2) edge[bend right=50] coordinate (8) (6)
	(3) edge[bend right=40] coordinate (9) (6);

\path
	(1) edge[out=-60,in=180] (7)
	(2) edge[out=-50,in=180] (8)
	(3) edge[out=-40,in=180] (9);

\path[->,dashed,shorten >=2pt]
	(7) edge[out=0,in=240] (6)
	(8) edge[out=0,in=230] (6)
	(9) edge[out=0,in=220] (6);

\path[freccia,dashed] (4) edge[bend right,dashed] (6);
\end{tikzpicture}
\end{center}
\caption{The graph $\Sigma_n$.}
\label{fig:sphere}

\bigskip

\begin{center}
\begin{tikzpicture}[>=stealth,node distance=2cm,
main node/.style={circle,inner sep=2pt},
freccia/.style={->,shorten >=2pt,shorten <=2pt},
ciclo/.style={out=130, in=50, loop, distance=2cm}]

\clip (-0.9,-0.5) rectangle (10.9,1.4);

      \node[main node] (1) {};
      \node (2) [main node,right of=1] {};
      \node (3) [main node,right of=2] {};
      \node (4) [main node,right of=3] {};
      \node (5) [right of=4] {};
      \node (6) [main node,right of=5] {};

      \filldraw (1) circle (0.06) node[below left] {$w_0$};
      \filldraw (2) circle (0.06);
      \filldraw (3) circle (0.06);
      \filldraw (4) circle (0.06);
      \filldraw (6) circle (0.06) node[below right] {$w_n$};

      \path[freccia] (1) edge[ciclo] (1);
      \path[freccia] (2) edge[ciclo] (2);
      \path[freccia] (3) edge[ciclo] (3);
      \path[freccia] (4) edge[ciclo] (4);
      \path[freccia] (6) edge[ciclo] (6);

      \path[freccia] (1) edge (2) (2) edge (3) (3) edge (4);
      \path[freccia,dashed] (4) edge (6);
\end{tikzpicture}\vspace{-10pt}
\end{center}
\caption{The graph $\widetilde{\Sigma}_n$.}
\label{fig:sphereB}
\end{figure}

\begin{thm}\label{thm:nonUiso}
A (non $U(1)$-equivariant) isomorphism $\psi:C^*(\widetilde{\Sigma}_n)\to C^*(\Sigma_n)$ is given by
\begin{equation}\label{eq:psi}
\psi(P_{w_i})=P_{v_i} \;,\qquad
\psi(S_{e_i})=S_{e_{i,i}} \;,\qquad
\psi(S_{f_j})=\sum_{k=j+1}^nS_{e_{j,k}}S_{e_{j+1,k}}^* \;,
\end{equation}
for all $0\leq i\leq n$ and $0\leq j<n$. The inverse map is given on generators by
\[
\psi^{-1}(P_{v_i})=P_{w_i} \;,\qquad
\psi^{-1}(S_{e_{i,i}})=S_{e_i} \;,\qquad
\psi^{-1}(S_{e_{j,k}})=S_{f_j}S_{f_{j+1}}\ldots S_{f_{k-1}}S_{e_k} \;,
\]
for all $0\leq i\leq n$ and for all $0\leq j<k\leq n$.
\end{thm}

\begin{proof}
Using the Cuntz-Krieger relations and \eqref{eq:heavilyuse} one easily checks that $\psi$ and $\psi^{-1}$ are well-defined and one the inverse of the other.
For example, let us verify that $\psi$ is well-defined, i.e.~that the elements in \eqref{eq:psi} form a Cuntz-Krieger $\widetilde{\Sigma}_n$-family.
For all $0\leq i\leq n$:
\[
\psi(S_{e_i})^*\psi(S_{e_i})=S_{e_{i,i}}^*S_{e_{i,i}}=P_{v_i}=\psi(P_{w_i}) .
\]
For all $0\leq j<n$:
\begin{align}
\psi(S_{f_j})^*\psi(S_{f_j})
&=\sum_{k,k'=j+1}^nS_{e_{j+1,k}}S_{e_{j,k}}^*S_{e_{j,k'}}S_{e_{j+1,k'}}^*  \notag \\
&=\sum_{k=j+1}^nS_{e_{j+1,k}}P_{v_k}S_{e_{j+1,k}}^*
=\sum_{k=j+1}^nS_{e_{j+1,k}}S_{e_{j+1,k}}^*=P_{v_{j+1}}=\psi(P_{w_{j+1}}) \label{eq:fromabove}
\end{align}
and
\begin{multline*}
\psi(S_{e_j})\psi(S_{e_j})^*+\psi(S_{f_j})\psi(S_{f_j})^*
=S_{e_{j,j}}S_{e_{j,j}}^*+\sum_{k,k'=j+1}^nS_{e_{j,k'}}S_{e_{j+1,k'}}^*S_{e_{j+1,k}}S_{e_{j,k}}^* \\
=S_{e_{j,j}}S_{e_{j,j}}^*+\sum_{k=j+1}^nS_{e_{j,k}}P_{v_k}S_{e_{j,k}}^*=\sum_{k=j}^nS_{e_{j,k}}S_{e_{j,k}}^*=P_{v_j}=\psi(P_{w_j}) .
\end{multline*}
Finally
\[
\psi(S_{e_n})\psi(S_{e_n})^*=S_{e_{n,n}}S_{e_{n,n}}^*=P_{v_n}=\psi(P_{w_n}) .
\]
Thus \eqref{eq:CK1p} and \eqref{eq:CK2p} for the graph $\widetilde{\Sigma}_n$ are satisfied.
Clearly the elements $\psi(P_{w_i})$ are orthogonal projections and $\psi(S_{e_i})$ is a partial isometry.
From \eqref{eq:fromabove} we get
\[
\psi(S_{f_j})\psi(S_{f_j})^*\psi(S_{f_j})=\psi(S_{f_j})\psi(P_{w_{j+1}})=\sum_{k=j+1}^nS_{e_{j,k}}S_{e_{j+1,k}}^*P_{v_{j+1}}=
\psi(S_{f_j}) ,
\]
so that $\psi(S_{f_j})$ is a partial isometry as well. This proves that $\psi$ is well-defined.

The proof that $\psi^{-1}$ is well-defined and is the inverse of $\psi$ is analogous and it is  omitted.
\end{proof}

\begin{rem}
Since all the formulas in Theorem~\ref{thm:nonUiso} are purely algebraic, the pair $(\psi,\psi^{-1})$ restricts to an isomorphism $L_{\C}(\widetilde{\Sigma}_n)\cong L_{\C}(\Sigma_n)$ between Leavitt path algebras.
\end{rem}

We now pass to quantum spheres.

\begin{prop}[The universal C*-algebra of a quantum sphere]\label{prop:embedded}
For all $0\leq q<1$ the universal C*-algebra of $\mathcal{A}(S^{2n+1}_q)$ exists and will be denoted by $C(S^{2n+1}_q)$.
If $q\neq 0$, the canonical \mbox{*-homomorphism} from $\mathcal{A}(S^{2n+1}_q)$ to its universal C*-algebra $C(S^{2n+1}_q)$ is injective.
\end{prop}

\begin{proof}
Because of \eqref{eq:qsphereE}, in any bounded *-representation the norm of the $z$'s is bounded by one, and then the norm of any element $a\in\mathcal{A}(S^{2n+1}_q)$ is bounded by a constant $K_a$ which is independent of the representation.
From this observation, it follows that the universal \mbox{C*-seminorm} is well-defined on $\mathcal{A}(S^{2n+1}_q)$ and the C*-enveloping algebra exists.
If $q\neq 0$, since $\mathcal{A}(S^{2n+1}_q)$ has a faithful representation (Prop.~\ref{prop:faithrep}), the universal C*-seminorm is a norm, hence injectivity of the map $\mathcal{A}(S^{2n+1}_q)\to C(S^{2n+1}_q)$.
\end{proof}

We now prove that the canonical \mbox{*-homomorphism} $\mathcal{A}(S^{2n+1}_q)\to C(S^{2n+1}_q)$ is injective
for $q=0$ as well, cf.~the next theorem.

\begin{thm}\label{thm:graphq0}
~
\begin{enumerate}
\item
The canonical *-homomorphism from $\mathcal{A}(S^{2n+1}_0)$ to its C*-enveloping algebra $C(S^{2n+1}_0)$ is injective.

\item
A $U(1)$-equivariant isomorphism $\varphi:L_{\C}(\Sigma_n)\to\mathcal{A}(S^{2n+1}_0)$ is given by
\begin{equation}\label{eq:CKa}
\varphi(P_{v_i})=z_iz_i^* \;,\qquad
\varphi(S_{e_{i,j}})=z_iz_jz_j^* ,
\end{equation}
for all $0\leq i\leq j\leq n$. Its inverse is given on generators by
\[
\varphi^{-1}(z_i)=\sum\nolimits_{j=i}^nS_{e_{i,j}} ,
\]
for all $0\leq i\leq n$.

\item
The map $\varphi$ extends to an isomorphism of C*-algebras $\widetilde{\varphi}:C^*(\Sigma_n)\to C(S^{2n+1}_0)$.

\item\label{en:qzeroB}
The representation $\pi$ of Prop.~\ref{prop:rep} extends to a faithful representation of $C(S^{2n+1}_0)$.
\end{enumerate}
\end{thm}

\begin{proof}
First, we show that the elements in \eqref{eq:CKa} satisfy the relations (\ref{eq:CK0}-\ref{eq:CK4}) for the graph $\Sigma_n$, so that the maps $\varphi$ and $\widetilde{\varphi}$ are well-defined. The relation \eqref{eq:CK0} follows from Prop.~\ref{prop:Zproj}. We will use repeatedly (and tacitly) Prop.~\ref{prop:Zproj}.

Let $i\leq j$ and $k\leq l$. If $i\neq k$:
\[
\varphi(S_{e_{i,j}})^*\varphi(S_{e_{k,l}})=(z_jz_j^*)(z_i^*z_k)(z_lz_l^*)\stackrel{\eqref{eq:qsphereBp}}{=}0 .
\]
On the other hand, if $i=k$ we get:
\begin{multline*}
\varphi(S_{e_{i,j}})^*\varphi(S_{e_{i,l}})=(z_jz_j^*)(z_i^*z_i)(z_lz_l^*)\stackrel{\eqref{eq:qsphereCp}}{=}\sum\nolimits_{s\geq i}(z_jz_j^*)(z_sz_s^*)(z_lz_l^*) \\
=\sum\nolimits_{s\geq i}\delta_{j,s}\delta_{s,l}z_sz_s^*=\delta_{j,l}\varphi(P_{v_j}) .
\end{multline*}
This proves \eqref{eq:CK1}. Next, for all $i\leq j$:
\begin{multline*}
\sum\nolimits_{j\geq i}\varphi(S_{e_{i,j}})\varphi(S_{e_{i,j}})^*=\sum\nolimits_{j\geq i}z_i(z_jz_j^*)^2z_i^*=\sum\nolimits_{j\geq i}z_iz_jz_j^*z_i^* \\
\stackrel{\eqref{eq:qsphereAp}}{=}\sum\nolimits_{j=0}^nz_iz_jz_j^*z_i^*\stackrel{\eqref{eq:qsphereEp}}{=}z_iz_i^*=\varphi(P_{v_i}) .
\end{multline*}
This proves \eqref{eq:CK2}. Finally,
\[
\varphi(P_{v_i})\varphi(S_{e_{i,j}})=z_i(z_i^*z_i)z_jz_j^*\stackrel{\eqref{eq:qsphereCp}}{=}z_i\sum\nolimits_{s\geq i}(z_sz_s^*)(z_jz_j^*)=z_i(z_jz_j^*)=\varphi(S_{e_{i,j}}) .
\]
In a similar way one proves that $\varphi(S_{e_{i,j}})\varphi(P_{v_j})=\varphi(S_{e_{i,j}})$, which proves \eqref{eq:CK3}. By adjunction one obtains \eqref{eq:CK4}.

This proves that the *-homomorphisms $\varphi$ and $\widetilde{\varphi}$ are well-defined, and we have a commutative diagram
\[
\begin{tikzpicture}[>=To,xscale=3,yscale=2]

\node (a) at (0,1) {$C^*(\Sigma_n)$};
\node (b) at (1,1) {$C(S^{2n+1}_0)$};
\node (c) at (0,0) {$L_{\C}(\Sigma_n)$};
\node (d) at (1,0) {$\mathcal{A}(S^{2n+1}_0)$};

\draw[font=\footnotesize,->]
	(c) edge node[above] {$\varphi$} (d)
	(d) edge node[right] {$\iota$} (b)
	(a) edge node[above] {$\widetilde{\varphi}$} (b)
	(c) edge[right hook->] (a);

\end{tikzpicture}
\]
where the left vertical arrow is the inclusion of $L_{\C}(\Sigma_n)$ into $C^*(\Sigma_n)$ and $\iota$ is the canonical \mbox{*-homomorphism} from $\mathcal{A}(S^{2n+1}_0)$ to its C*-enveloping algebra.

Let $\pi$ be the representation in Prop.~\ref{prop:rep}. By universality of the C*-algebra, there exists a representation $\widetilde{\pi}$ of $C(S^{2n+1}_0)$ on the same Hilbert space such that $\widetilde{\pi}\circ\iota=\pi$. Since
$\pi(z_iz_i^*)=Z_iZ_i^*\neq 0$, one also has $\widetilde{\varphi}(P_{v_i})=\iota(z_iz_i^*)\neq 0$. It follows from Theorem~\ref{thm:gut} that $\widetilde{\varphi}$ is injective, and then $\varphi$ is injective, and $\iota$ is injective on the image of $\varphi$.

From \eqref{eq:qsphereAp} we get
\[
\varphi\Big(\sum_{j=i}^nS_{e_{i,j}}\Big)=
z_i\sum_{j=i}^nz_jz_j^*=
z_i\sum_{j=0}^nz_jz_j^*=z_i ,
\]
proving that $\varphi$ is surjective, hence an isomorphism $L_{\C}(\Sigma_n)\cong\mathcal{A}(S^{2n+1}_0)$, and $\iota$ is injective.
It also follows that the image of $\widetilde{\varphi}:C^*(\Sigma_n)\to C(S^{2n+1}_0)$ contains $\iota\big(\mathcal{A}(S^{2n+1}_0)\big)$ which is dense in $C(S^{2n+1}_0)$. But *-homomorphisms of C*-algebras have closed image, hence $\widetilde\varphi$ is surjective.

Finally, $\widetilde{\pi}\circ\widetilde{\varphi}$ is a $U(1)$-covariant representation of $C^*(\Sigma_n)$ that maps each $P_{v_i}$ to a non-zero operator. By Theorem~\ref{thm:gut} this representation is faithful, and then $\widetilde{\pi}$ is faithful as well.
\end{proof}

From now on, we will think of $\mathcal{A}(S^{2n+1}_0)$ as a dense *-subalgebra of $C(S^{2n+1}_0)$ and omit $\iota$,
and denote by the same symbol $\pi$ both the representation in Prop.~\ref{prop:rep} and its extension to $C(S^{2n+1}_0)$.

\begin{prop}\label{prop:strongly}
For all $0\leq q<1$, the actions of $\mathbb{T}^{n+1}$ and $U(1)$ on $\mathcal{A}(S^{2n+1}_q)$ extend to strongly continuous actions on $C(S^{2n+1}_q)$.
\end{prop}

\begin{proof}
The $U(1)$-action is the composition of the action $\alpha$ of $\mathbb{T}^{n+1}$ on $\mathcal{A}(S^{2n+1}_q)$ with the diagonal embedding $U(1)\to\mathbb{T}^{n+1}$, hence it is enough to prove the claim for the action $\alpha$ in \eqref{eq:Taction}.

If $a\in\mathcal{A}(S^{2n+1}_q)$ is a monomial as in \eqref{eq:basis}, from \eqref{eq:actmon} it follows that the map
\[
f_a:\mathbb{T}^{n+1}\to\mathcal{A}(S^{2n+1}_q) , \qquad t\mapsto\alpha_{\mv{t}}(a),
\] 
is continuous. By linearity, the map $f_a$ is then continuous for all $a\in\mathcal{A}(S^{2n+1}_q)$.
From Prop.~\ref{prop:contact} it follows that $\alpha$ extends to a strongly continuous action on $C(S^{2n+1}_q)$.
\end{proof}

\subsection{The quantum sphere with positive parameter}\label{sec:6}
In this section, $q\neq 0$.
Let $Z_0,\ldots,Z_n$ be the bounded operators on $\ell^2(\N^n\times\Z)$ in Prop.~\ref{prop:rep}
and notice that, by Theorem~\ref{thm:graphq0}\ref{en:qzeroB} they generate a C*-subalgebra of $\mathcal{B}(\ell^2(\N^n\times\Z))$ that is isomorphic
to $C(S^{2n+1}_0)$, and that will be identified with $C(S^{2n+1}_0)$ in this section.

By the universal property of C*-enveloping algebras, the representation in Prop.~\ref{prop:rep} extends to a representation $\pi:C(S^{2n+1}_q)\to\mathcal{B}(\ell^2(\N^n\times\Z))$, that we denote by the same symbol. We will prove first that
the image of this representation is $C(S^{2n+1}_0)$, and then that it is injective, thus proving that $\pi$ induces an isomorphism $C(S^{2n+1}_q)\cong C(S^{2n+1}_0)$.

\begin{lemma}\label{lemma:sub}
For all $0<q<1$, one has $\pi\big(C(S^{2n+1}_q)\big)\subseteq C(S^{2n+1}_0)$.
\end{lemma}

\begin{proof}
For $0\leq i<n$ and $k_0,\ldots,k_i\in\N^{i+1}$, define
\begin{equation}\label{eq:Toper}
T(k_0,\ldots,k_i):=Z_0^{k_0}Z_1^{k_1}\cdots Z_i^{k_i} .
\end{equation}
These are the operators in \cite{HS02}, equation (4.6). Then, define
\[
A_i:=\sum_{k_0,\ldots,k_i\in\N}
q^{k_0+\ldots+k_{i-1}}\left(
\sqrt{1-q^{2k_i+2}}-\sqrt{1-q^{2k_i}}
\right)T(k_0,\ldots,k_i)Z_iT(k_0,\ldots,k_i)^*
\]
if $0\leq i<n$, and
\[
A_n:=\sum_{k_0,\ldots,k_{n-1}\in\N}
q^{k_0+\ldots+k_{n-1}}T(k_0,\ldots,k_{n-1})Z_nT(k_0,\ldots,k_{n-1})^* .
\]
First, we will prove that $A_i\in C(S^{2n+1}_0)$. Next, we will prove that $A_i=\pi(z_i)$, which implies $\pi(z_i)\in C(S^{2n+1}_0)$ (for all $0\leq i\leq n$), and therefore $\pi\big(C(S^{2n+1}_q)\big)\subseteq C(S^{2n+1}_0)$.

\medskip

\noindent
{\bf Step 1.}
Since each $Z_i$ is a partial isometry, the operators \eqref{eq:Toper} have norm $\leq 1$.
For all $q,x\in\ointerval{0}{1}$ one has
\[
0<\sqrt{1-q^2x^2}-\sqrt{1-x^2}<x .
\]
Using this for $x=q^{k_i}$ one proves that 
the $(k_0,\ldots,k_i)$ term in the sum defining $A_i$ has norm bounded by $q^{k_0+\ldots+k_i}$.
Thus the series is norm-convergent (for $|q|<1$, $\sum_{k\in\N}q^{k}$ is a convergent geometric series)
and $A_i$ is a well-defined element of $C(S^{2n+1}_0)$.

\medskip

\noindent
{\bf Step 2.}
Since $(Z_0^*)^{k_0}\ket{j_0,\ldots,j_n}=0$ for $j_0<k_0$, we have
\begin{align*}
A_0\ket{j_0,\ldots,j_n}
&=\sum_{k_0=0}^{j_0}\left(\sqrt{1-q^{2k_0+2}}-\sqrt{1-q^{2k_0}}\right)Z_0^{k_0+1}(Z_0^*)^{k_0}\ket{j_0,\ldots,j_n} \\
&=\bigg(\sum_{k_0=0}^{j_0}\left(\sqrt{1-q^{2k_0+2}}-\sqrt{1-q^{2k_0}}\right)\bigg)\ket{j_0+1,j_1,\ldots,j_n} \\
&=\sqrt{1-q^{2j_0+2}}\ket{j_0+1,j_1,\ldots,j_n} ,
\end{align*}
where we observed that we have a telescoping sum in the second line.

Next, let $0<i<n$.  By counting the powers of $Z^*$'s (each operator being a negative shift),
we see that $T(k_0,\ldots,k_i)^*\ket{j_0,\ldots,j_n}$ is zero unless $k_m\leq j_m$ for all $0\leq m\leq i$.
If these inequalities are satisfied, then
\[
Z_iT(k_0,\ldots,k_i)^*\ket{j_0,\ldots,j_n}=
\delta_{j_0-k_0,0}\ldots\delta_{j_{i-1}-k_{i-1},0}\ket{0,\ldots,0,j_i-k_i+1,j_{i+1},\ldots,j_n} .
\]
Thus
\[
T(k_0,\ldots,k_i)Z_iT(k_0,\ldots,k_i)^*\ket{j_0,\ldots,j_n}=\delta_{k_0,j_0}\ldots\delta_{k_{i-1},j_{i-1}}
\ket{j_0,\ldots,j_i+1,\ldots,j_n} ,
\]
and finally
\[
A_i\ket{j_0,\ldots,j_n}=\sum_{k_i=0}^{j_i}
q^{j_0+\ldots+j_{i-1}}\left(
\sqrt{1-q^{2k_i+2}}-\sqrt{1-q^{2k_i}}
\right)\ket{j_0,\ldots,j_i+1,\ldots,j_n} .
\]
Since the latter is a telescoping sum, we get the desired result
\[
A_i\ket{j_0,\ldots,j_n}=q^{j_0+\ldots+j_{i-1}}\sqrt{1-q^{2j_i+2}}\ket{j_0,\ldots,j_i+1,\ldots,j_n} .
\]
Finally, for $i=n$,
\[
Z_nT(k_0,\ldots,k_{n-1})^*\ket{j_0,\ldots,j_n}=
\delta_{j_0-k_0,0}\ldots\delta_{j_{n-1}-k_{n-1},0}\ket{0,\ldots,0,0,j_n+1} ,
\]
and then
\[
A_n\ket{j_0,\ldots,j_n}=q^{j_0+\ldots+j_{n-1}}\ket{j_0,\ldots,j_{n-1},j_n+1} .
\]
Comparing these with the formulas in Prop.~\ref{prop:rep}, we see that $A_i=\pi(z_i)$ for all $0\leq i\leq n$.
\end{proof}

\begin{lemma}\label{lemma:sup}
For all $0<q<1$, one has $\pi\big(C(S^{2n+1}_q)\big)\supseteq C(S^{2n+1}_0)$.
\end{lemma}

\begin{proof}
Let $A$ be a concrete C*-algebra of bounded operators on a Hilbert space, $a=a^*\in A$, $\lambda\in\R$ and $\chi_\lambda:\R\to\R$ the characteristic function of the singleton $\{\lambda\}$. If $\lambda$ is an isolated point in the spectrum, then $\chi_\lambda(a)$ belongs to $A$ and is a projection onto the eigenspace corresponding to the eigenvalue $\lambda$.
Now, let $A:=\pi\big(C(S^{2n+1}_q)\big)$. We must show that $Z_i\in A$ for all $0\leq i\leq n$.

For all $1\leq i\leq n$, call
\begin{equation}\label{eq:ri}
\xi_i:=\sum_{j=i}^nz_jz_j^*  .
\end{equation}
It follows from \eqref{eq:diagonal} that
\[
\pi(\xi_i)\ket{k_0,\ldots,k_n}=q^{2(k_0+\ldots+k_{i-1})}\ket{k_0,\ldots,k_n} .
\]
For diagonal operators, the spectrum is the closure of the point spectrum, and the only accumulation point in the spectrum of $\pi(\xi_i)$ is $0$.
It follows from the discussion above that $y_i:=\chi_1(\pi(\xi_i))$ belongs to $A$, and
\[
y_i\ket{k_0,\ldots,k_n}=\delta_{k_0,0}\ldots\delta_{k_{i-1},0}\ket{0,\ldots,0,k_i,\ldots,k_n} .
\]
In particular, $Z_n=\pi(z_n)y_n\in A$.

Next, for all $0\leq i<n$ the operator $\big\{1-\pi(z_i^*z_i)\big\}y_i$ is diagonal and all non-zero eigenvalues are isolated points in the spectrum. Each eigenprojection $p_{i,m}$, given for $m\in\N$ by
\[
p_{i,m}\ket{k_0,\ldots,k_n}=\delta_{k_0,0}\ldots\delta_{k_{i-1},0}\delta_{k_i,m}\ket{0,\ldots,0,m,k_{i+1},\ldots,k_n} ,
\]
belongs to $A$. Let
\[
a_i:=\sum_{m\in\N}\frac{1}{\sqrt{1-q^{2(m+1)}}}\,p_{i,m} .
\]
This is a diagonal operator,
\[
a_i\ket{k_0,\ldots,k_n} =\frac{1}{\sqrt{1-q^{2(k_i+1)}}}
\delta_{k_0,0}\ldots\delta_{k_{i-1},0}\ket{0,\ldots,0,k_i,\ldots,k_n} .
\]
The partial sums $s_j:=\sum_{m\leq j}\frac{1}{\sqrt{1-q^{2(m+1)}}}\,p_{i,m}$ form a Cauchy sequence. Indeed, for all $\varepsilon>0$ small enough, choose an integer $N_\varepsilon$ big enough so that $(1-q^{2N_\varepsilon})^{-1/2}<\varepsilon$. Then
\[
\|s_i-s_j\|\leq\sup\left\{(1-q^{2(m+1)})^{-1/2}:j<m\leq i\right\}<\varepsilon ,\qquad\forall\;i>j>N_\varepsilon .
\]
Thus $a_i\in A$, 
which implies $Z_i=\pi(z_i)a_i\in A$.
\end{proof}

Denote by $D_n$ the commutative unital C*-subalgebra of $C(S^{2n+1}_q)$ generated by the set \eqref{eq:mutuallycommuting}. Clearly $D_0=\C 1\subseteq C(S^1)$. We are interested in the case $n\geq 1$.
A more convenient set of generators for $D_n$, that we are going to use below, is given by the elements in \eqref{eq:ri}.
By Gelfand duality, one has $D_n\cong C(\Omega^n_q)$, where
\[
\Omega^n_q:= \Big\{ \big(\chi(\xi_1),,\ldots,\chi(\xi_n)\big)\in\R^n :\chi\text{ is a character of }D_n\Big\}
\]
is the spectrum of $D_n$ (and $C(\Omega^n_q)$ denotes the C*-algebra of continuous functions on this topological space). We stress that $\Omega^n_q$ is a standard topological space, and not a ``quantum space''.
The subscript $q$ refers to the fact that this space depends on the value of the parameter, as we shall see.
Notice also that $\Omega^n_q$ coincides with the joint spectrum
\[
\Omega^n_q =\sigma\big( \xi_1, \ldots,\xi_n \big) ,
\]
i.e.~with the set of $(\lambda_1,\ldots,\lambda_n)\in\R^n$ such that
the ideal in $D_n$ spanned by $\lambda_1-\xi_1,\ldots,\lambda_n-\xi_n$ is proper.

\begin{lemma}\label{lemma:YYs}
Let $0<t<1$ and let $Y$ be an element in a C*-algebra satisfying
\begin{equation}\label{eq:satisfyingYY}
Y^*Y=tYY^*+1-t .
\end{equation}
Then,
\begin{equation}\label{eq:wederive}
\sigma(YY^*)\subseteq\big\{1-t^k:k\in\overline{\N}\big\} .
\end{equation}
\end{lemma}

\begin{proof}
Let $\lambda\in\sigma(Y^*Y)$ and assume that
\begin{equation}\label{eq:assume}
\lambda\neq 1 .
\end{equation}
For $n\in\N$, define
\[
\mu_n:=t^{-n}(\lambda-1)+1 .
\]
By contradiction, assume that $\mu_n\neq 0$ for all $n\in\N$.
Since $\mu_{n+1}=t^{-1}\mu_n+1-t^{-1}$, from \eqref{eq:satisfyingYY} we deduce that
\[
\mu_n\in\sigma(Y^*Y)\Longrightarrow\mu_{n+1}\in\sigma(YY^*).
\]
But
\begin{equation}\label{eq:wededuce}
\sigma(YY^*)\cup\{0\}=\sigma(Y^*Y)\cup\{0\} ,
\end{equation}
hence 
\[
\mu_n\in\sigma(Y^*Y)\Longrightarrow\mu_{n+1}\in\sigma(Y^*Y).
\]
Since $\mu_0=\lambda\in\sigma(Y^*Y)$, we deduce by induction that $\mu_n\in\sigma(Y^*Y)$
for all $n\in\N$. But the assumption \eqref{eq:assume} implies that the sequence $(\mu_n)_{n\in\N}$ is divergent, and we get a contradiction (the spectrum is bounded by the norm of $Y$).
Hence, it must be $\mu_n=0$ for some $n\in\N$, that means $\lambda=1-t^n$.
All the points $\lambda$ in $\sigma(Y^*Y)$ different from $1=1-t^{+\infty}$ are of the above form, thus $\sigma(Y^*Y)\subseteq\big\{1-t^k:k\in\overline{\N}\big\}$ and, using again \eqref{eq:wededuce}, we get \eqref{eq:wederive}.
\end{proof}

In the following, we denote by $\vartheta$ the faithful conditional expectation on $C(S^{2n+1}_q)$ described in Lemma~\ref{lemma:faithcond}\ref{en:faithcondA} and extending the map in \eqref{eq:maptheta}, which is denoted by the same symbol.

\begin{prop}\label{prop:thetabis}
For all $0<q<1$, one has
\[
D_n=\vartheta(C(S^{2n+1}_q))=C(S^{2n+1}_q)^{\mathbb{T}^{n+1}} .
\]
\end{prop}

\begin{proof}
The second equality is obvious. From Prop.~\ref{prop:theta}, one has $\vartheta(\mathcal{A}(S^{2n+1}_q))=\mathcal{A}(\Delta^n_q)$ and then by continuity $\vartheta(C(S^{2n+1}_q))=D_n$, since $D_n$ is the norm-closure of $\mathcal{A}(\Delta^n_q)$.
\end{proof}

The next lemma is Lemma 4.1 of \cite{HS02}. We follow the proof of \cite[Lemma 2.10]{MK22}.

\begin{lemma}\label{lemma:one}
Let $q\neq 0$, $n\geq 1$, and let $F_n:\overline{\N}^n\to\R^n$ be the closed embedding
\[
F_n(k_1,\ldots,k_n):=(q^{2k_1},q^{2(k_1+k_2)},\ldots,q^{2(k_1+\ldots+k_n)}) .
\]
Then, $\Omega^n_q=\mathrm{Im}(F_n)$.
\end{lemma}

\begin{proof}
The joint spectrum contains the joint eigenvalues of all bounded *-representations and their limit points as well.
From Prop.~\ref{prop:joint}\ref{en:jointB}, the set of joint eigenvalues of $(\pi(\xi_1),\ldots,\pi(\xi_n))$ is given by all tuples of the form $F_n(k_1,\ldots,k_n)$ with $k_1,\ldots,k_n\in\N$ (here we label the components of the tuple from $1$ to $n$ rather than from $0$ to $n-1$). Thus, $F_n(\N^n)\subseteq \Omega^n_q$ and, since $\Omega^n_q$ is closed,
$F_n(\overline{\N}^n)\subseteq \Omega^n_q$. We must prove the opposite inclusion.

\smallskip

Let $\chi$ be a character of $D_n$. From \eqref{eq:relk} with $i=0$ and $m=1$ we see that $Y=z_0$ satisfies the hypothesis of Lemma \ref{lemma:YYs} with $t=q^2$. Hence
\[
1-\chi(\xi_1)=\chi(z_0z_0^*)\in\sigma(z_0z_0^*)\subseteq\big\{1-q^{2k}:k\in\overline{\N}\big\} .
\]
Thus, $\chi(\xi_1)=q^{2k}$ for some $k\in\overline{\N}$.
Now we show that, for all $1\leq i<n$, if $\chi(\xi_i)=q^{2k}$ with $k\in\overline{\N}$, then $\chi(\xi_{i+1})=q^{2(k+m)}$ for some $m\in\overline{\N}$, thus concluding the proof by induction.

\smallskip

From now on we assume that $1\leq i<n$ is fixed, $k\in\overline{\N}$ is fixed, and $\chi(\xi_i)=q^{2k}$.

Denote by $A\subseteq C(S^{2n+1}_q)$ the unital C*-subalgebra generated by $z_i,\ldots,z_n$ and let $C\subseteq A$ be the commutative unital C*-subalgebra generated by $z_iz_i^*,\ldots,z_nz_n^*$. We note that
\begin{itemize}\itemsep=2pt
\item $\xi_i\in C\subseteq A$,
\item $\xi_i=1-\sum_{j=0}^{i-1}z_jz_j^*$ is a central element in $A$, cf.~Lemma \ref{lemma:commut}\ref{en:commutA}.
\end{itemize}
We call $I$ the two-sided *-ideal in $A$ generated by $\xi_i-q^{2k}$ and
$J\subseteq I$ the two-sided *-ideal in $C$ generated by $\xi_i-q^{2k}$.
Thus, we have a commutative diagram
\[
\begin{tikzcd}[column sep=large,row sep=large]
&& C(S^{2n+1}_q) \\
0\ar[r] & I\ar[r] & A\ar[r,"p_1"]\ar[u,hook] & A/I\ar[r] & 0 \\
0\ar[r] & J\ar[r]\ar[u,hook] & C\ar[r,"p_2"]\ar[u,hook] & C/J\ar[r]\ar[u,"f",swap] & 0
\end{tikzcd}
\]
whose two rows are extensions of C*-algebras.

Let $\vartheta:C(S^{2n+1}_q)\to D_n$ be the conditional expectation as before.
One can repeat the proof of Prop.~\ref{prop:theta} to show that a polynomial in $z_i,\ldots,z_n$ is $\mathbb{T}^{n+1}$-invariant if and only if it belongs to $C$, hence
$\vartheta$ restricts to a faithful conditional expectation $\vartheta|_A:A\to C$ (that we denote by the same symbol).
Since $\xi_i-q^{2k}$ is central in $A$, every element in $I$ can be written in the form $a(\xi_i-q^{2k})$ with $a\in A$. Since $\xi_i-q^{2k}$ is $\mathbb{T}^{n+1}$-invariant,
\begin{align*}
\vartheta(a(\xi_i-q^{2k})) &=\int_{\mathbb{T}^{n+1}}\alpha_{\mv{t}}(a)\alpha_{\mv{t}}(\xi_i-q^{2k})d\mu_t \\
&=\left(\int_{\mathbb{T}^{n+1}}\alpha_{\mv{t}}(a)d\mu_t\right)(\xi_i-q^{2k})
=\vartheta(a)(\xi_i-q^{2k}) ,
\end{align*}
proving that $\vartheta(I)\subseteq J$. Hence, $\vartheta$ induces a linear map $[\vartheta]:A/I\to C/J$ between quotient C*-algebras. From the above diagram we get the following commutative diagram
\[
\begin{tikzcd}[column sep=large,row sep=large]
0\ar[r] & I\ar[r]\ar[d,bend right,dashed,"\vartheta|_I",swap] & A\ar[r,"p_1"]\ar[d,bend right,dashed,"\vartheta|_A",swap] & A/I\ar[r]\ar[d,bend right,dashed,"{[\vartheta]}",swap] & 0 \\
0\ar[r] & J\ar[r]\ar[u] & C\ar[r,"p_2"]\ar[u] & C/J\ar[r]\ar[u,"f",swap] & 0
\end{tikzcd}
\]
For all $c\in C$, since $\vartheta|_A$ is the identity on $C$,
\[
([\vartheta]\circ f)(p_2(c))=[\vartheta](p_1(c))=p_2(\vartheta|_A(c))=p_2(c) .
\]
Since $p_2$ is surjective, $[\vartheta]\circ f=\id_{C/J}$, which implies that $f$ is injective.
From now on, we will identify $C/J$ with the C*-subalgebra $f(C/J)$ of $A/I$.

In the quotient algebra $A/I$ one has
\begin{equation}\label{eq:inthequotient}
p_1(\xi_{i+1})=p_1(\xi_i)-p_1(z_i)p_1(z_i)^*=q^{2k}-p_1(z_i)p_1(z_i)^*
\end{equation}
and from \eqref{eq:qsphereC},
\[
p_1(z_i)^*p_1(z_i)=p_1(z_i)p_1(z_i)^*+(1-q^2)p_1(\xi_{i+1})
=q^2\,p_1(z_i)p_1(z_i)^*+(1-q^2)q^{2k} .
\]
If $k\neq +\infty$, putting $Y:=q^{-k}p_1(z_i)$ and $t:=q^2$, then \eqref{eq:satisfyingYY} is satisfied, which proves that
\[
\sigma\big(p_1(z_i)p_1(z_i)^*\big)\subseteq\big\{q^{2k}(1-q^{2m}):m\in\overline{\N}\big\} ,
\]
and then
\begin{equation}\label{eq:bothcases}
\sigma\big(p_1(\xi_{i+1})\big)=q^{2k}-\sigma\big(p_1(z_i)p_1(z_i)^*\big)\subseteq\big\{q^{2(k+m)}:m\in\overline{\N}\big\} .
\end{equation}
If $k=+\infty$, the equality \eqref{eq:inthequotient} becomes $p_1(\xi_{i+1})=-p_1(z_iz_i^*)$, and since the left hand side is a positive operator and the right hand side is a negative operator, $p_1(\xi_{i+1})=0$. In both cases,
$k\in\N$ and $k=+\infty$, the inclusion \eqref{eq:bothcases} is valid.

Now, under the identification of $C/J$ with a C*-subalgebra of $A/I$, 
$p_1(\xi_{i+1})$ and $p_2(\xi_{i+1})$ are the same operator (but in a different C*-algebra). Since the spectrum doesn't depend on the ambient C*-algebra,
\[
\sigma\big(p_1(\xi_{i+1})\big)=\sigma\big(p_2(\xi_{i+1})\big) .
\]
Going back to our character $\chi$, by construction of the ideal one has \mbox{$\chi(J)=0$}. Thus, the character $\chi:D_n\to\C$ induces a character $[\chi]:C/J\to\C$. 
From
\[
\chi(\xi_{i+1})=[\chi]\big(p_2(\xi_{i+1})\big) \in \sigma\big(p_2(\xi_{i+1})\big)
\]
and from \eqref{eq:bothcases} we get the desired result.
\end{proof}

Let $x_i:\Omega^n_q\to\R$ be the $i$-th Cartesian coordinate of $\R^n$ restricted to $\Omega^n_q$. We will identify $D_n$ with $C(\Omega^n_q)$ through the map that sends $\xi_i$ to $x_i$. (It is not difficult to verity that the set $\Omega^n_q$ is homeomorphic to the quantized $n$-simplex $\Delta^n_q$ of Rem.~\ref{rem:qsimp}.)

\begin{prop}\label{prop:inj}
The representation $\pi$ in Prop.~\ref{prop:rep} is faithful, for all $0<q<1$.
\end{prop}

\begin{proof}
Let $A=C(S^{2n+1}_q)$, $G=\mathbb{T}^{n+1}$ and observe that $A^G=D_n=C(\Omega^n_q)$ (by Prop.~\ref{prop:thetabis}). It is enough to show that the restriction $\pi|_{A^G}=\pi|_{C(\Omega^n_q)}$ is faithful (cf.~Lemma \ref{lemma:faithcond}\ref{en:faithcondB}). This is an explicit computation. For all $f\in C(\Omega^n_q)$ and all $k_1,\ldots,k_n\in\N$ one has
\[
\pi(f)\ket{k_1,\ldots,k_n}=f(q^{2k_1},\ldots,q^{2(k_1+\ldots+k_n)})\ket{k_1,\ldots,k_n}
\]
If $\pi(f)=0$, then $f$ is zero on a dense subset of $\Omega^n_q$, which implies $f=0$.
\end{proof}

\begin{thm}\label{thm:indep}
$C(S^{2n+1}_q)\cong C(S^{2n+1}_0)$ for all $0<q<1$.
\end{thm}

\begin{proof}
Let $\pi:C(S^{2n+1}_q)\to\mathcal{B}(\ell^2(\N^n\times\Z))$ be the representation in Prop.~\ref{prop:rep}. From Lemma \ref{lemma:sub} and Lemma \ref{lemma:sup} it follows that the image of $\pi$ is $C(S^{2n+1}_0)$.
It follows from Prop.~\ref{prop:inj} that $\pi$ is injective, hence it induces an isomorphism $C(S^{2n+1}_q)\to C(S^{2n+1}_0)$.
\end{proof}

Thus, the C*-algebra $C(S^{2n+1}_q)$ is independent of the value of the deformation parameter, and isomorphic to a graph C*-algebra for all $q\in\rinterval{0}{1}$.
It is worth mentioning the different behaviour of coordinate algebras: for $n\geq 1$ and
$q,q'\in\interval{0}{1}$, one has $\mathcal{A}(S^{2n+1}_q)\cong\mathcal{A}(S^{2n+1}_{q'})$ if and only if $q=q'$ (see \cite{Dan24b} for a proof).

\smallskip

\begin{samepage}
\begin{center}
\textsc{Acknowledgements}
\end{center}

\noindent
I would like to thank P.M.~Hajac, J.~Kaad, G.~Landi, M.H.~Mikkelsen, and W.~Szyma{\'n}ski for the many enlightening discussions about quantum spheres and (graph) C*-algebras.
The Author is a member of INdAM-GNSAGA (Istituto Nazionale di Alta Matematica ``F.\ Severi'') -- Unit\`a di Napoli 
and of INFN -- Sezione di Napoli.
This work is partially supported by the University of Naples Federico II under the grant FRA 2022 \emph{GALAQ: Geometric and ALgebraic Aspects of Quantization}.
\end{samepage}

\end{document}